\documentclass[12pt]{extarticle}
\usepackage{amsmath, amsthm, amssymb, hyperref, color}
\usepackage{graphicx}
\usepackage{caption}
\usepackage{subcaption}
\usepackage{mathtools}
\usepackage{enumerate}
\usepackage{stackengine}
\usepackage{ dsfont }
\usepackage[all]{xypic}
\usepackage{verbatim}
\usepackage{tikz-cd}
\usepackage{upgreek}
\usepackage{chemfig, chemnum}
\usepackage{tikz}
\usepackage[lined,commentsnumbered,ruled]{algorithm2e}
\usepackage{caption}
\usepackage{subcaption}
\usepackage{float}
\usepackage{youngtab}
\oddsidemargin \evensidemargin
\newtheorem{theorem}{Theorem}[section]
\newtheorem{proposition}[theorem]{Proposition}
\newtheorem{lemma}[theorem]{Lemma}
\newtheorem{corollary}[theorem]{Corollary}
\newtheorem{conjecture}[theorem]{Conjecture}
\theoremstyle{definition}

\newtheorem{remark}[theorem]{Remark}

\newtheorem{example}[theorem]{Example}

\newcommand{\CC}{\mathbb{C} }
\newcommand{\ZZ}{\mathbb{Z}}
\newcommand{\NN}{\mathbb{N}}

\title{\textbf{On Algebraic Theta Divisors and Rational Solutions of the KP Equation}}

\author{{Daniele Agostini, T\"urk\"u \"Ozl\"um \c{C}elik,  and John B. Little}}

\date{}

\begin{document}
	
	\maketitle
	
	\begin{abstract}
		In this paper we classify the singular curves whose theta divisors in their generalized Jacobians are algebraic, meaning that they are cut out by polynomial analogs of theta functions. We also determine the degree of an algebraic theta divisor in terms of the singularities of the curve. Furthermore, we show a precise relation between such algebraic theta functions and the corresponding tau functions for the KP hierarchy.

	\end{abstract}
	
	\section{Introduction}
	
	The theta divisor associated to a smooth projective curve $C$ is a fundamental object, governing much of the geometry of the curve. This was already recognized by Riemann himself, and afterwards there have been countless studies on the relationship between the curve and the theta divisor.   The classic text \cite[Chapter VI]{ACGH} (also see the exercises there and the bibliography) provides an excellent overview, and \cite{GrHu} is a more recent reference. 
	
	Theta divisors also have an important role in mathematical physics. Indeed, this divisor is cut out in the Jacobian by a Riemann theta function $\theta(\mathbf{z})$, which can in turn be used to construct solutions to the  Kadomtsev-Petviashvili (KP) equation
	\begin{equation}\label{eq:kp} 
	\frac{\partial}{\partial x}\left(  4f_t - 6ff_x - f_{xxx} \right) = 3f_{yy} 
	\end{equation}
	that describes the evolution of waves in shallow water. 
	
	Finally, the theta divisor is also an example of a double translation hypersurface: this is a hypersurface with two distinct parametrizations as Minkowski sums of curves.  Their study goes back to Sophus Lie,  and modern treatments can be found in~\cite{Little} and in \cite{AgoCelStrStu}. 
	
	The reason behind the many properties of the theta divisor is that it can be parametrized in terms of the Abel-Jacobi map of the curve. Indeed, if we fix a point $P_0 \in C$ and a basis $\boldsymbol{\omega} = (\omega_1,\dots,\omega_g)$ of canonical differentials for $H^0(C,\omega_C)$, then the image of the map
	\begin{equation}\label{eq:abeljacobig-1}
	C^{(g-1)} \to \operatorname{Jac}(C) = H^0(C,\omega_C)^*/H_1(C,\mathbb{Z}), \qquad P_1+\dots+P_{g-1} \mapsto \sum_{i=1}^{g-1} \int_{P_0}^{P_i} \boldsymbol{\omega}  
	\end{equation}
	is exactly the theta divisor. 
	This parametrization can actually be extended to any reasonably singular curve $C$. More precisely, if $C$ is a Gorenstein, reduced and connected projective curve of arithmetic genus $g$, then one can define its generalized Jacobian $\operatorname{Jac}(C)$ and then construct an Abel-Jacobi map as in \eqref{eq:abeljacobig-1}. The (closure of the) image of this map is again a divisor in the generalized Jacobian which is  called the theta divisor of the singular curve $C$. This lifts to a hypersurface in the complex vector space $H^0(C,\omega_C)^*$, which is cut out by an analytic equation $\theta(\mathbf{z})$
	which we call the theta function associated to the curve $C$.
	
	Such theta divisors retain many properties seen in the smooth curve case. For example, the theta function again gives solutions to the KP equation, and a theorem of Lie and Wirtinger \cite{W} states that any double translation hypersurface arises from a possibly singular curve in this way.
	
	Roughly speaking, the more singular the underlying curve is, the less transcendental the corresponding theta divisor is.  At the far end of the spectrum from the hypersurfaces defined by Riemann theta functions, there are cases where the theta divisor of a singular curve is even an \emph{algebraic hypersurface}. This means that the generalized Jacobian coincides with the complex vector space $H^0(C,\omega_C)^*$ and that the theta divisor inside it is cut out by a \emph{polynomial} theta function $\theta(\mathbf{z})$.  In the case $g = 3$ there was a significant amount of interest classically in understanding all types of theta surfaces, that is the theta divisors corresponding to singular plane quartic curves. The early 20th century mathematician John Eiesland, in particular, made an extensive study of the theta divisors from rational plane quartic curves
	in \cite{Eies09}.   He also gave a complete classification of the types of theta divisors that are algebraic in \cite{Eies08}.  Some of Eiesland's examples are reconsidered in \cite{AgoCelStrStu}.  
	In this paper, we will continue this study of algebraic theta divisors and theta functions in all dimensions.
	
	In particular, in Proposition \ref{AlgCharac}, we classify  all irreducible and reduced Gorenstein curves with an algebraic theta divisor. These are exactly those curves that are rational and all whose singularities are unibranch or analytically irreducible, meaning that they correspond to a unique point on the normalization of the curve. In particular, to each singular point we can associate its value semigroup and a corresponding partition $\lambda_{P}$, see Section  \ref{sec:weierstrassgap}. One of our main results is the computation of the degree of the theta divisor in terms of these data.
	
	\begin{theorem}\label{theoremDeg} 
		Let $C$ be an integral Gorenstein rational curve of arithmetic genus $g$, all of whose singular points of $C$ are unibranch. For each singular point $P$, let $\lambda_P$ be the corresponding partition. Then the degree of the corresponding algebraic theta function  is 
		\[ \sum_{P\in \operatorname{Sing}(C)} |\lambda_P|.\]
		In particular, this is less than or equal to $g(g+1)/2$ and this upper bound is attained if and only if the curve has exactly one unibranch singular point $P$ whose semigroup is 
		$\mathbb{S}_g =\langle 2,2g+1 \rangle$.
	\end{theorem}

	Another theme of our work is the connection with the KP equation: the algebraic theta functions arising from curves induce rational solutions to the KP equation see Corollary \ref{cor:rationalsolutionkp}. More generally, we extend the connection to the whole KP hierarchy: this is an infinite family of differential equations whose solutions are encoded by the so-called \emph{tau function}. We show in Proposition \ref{prop:pullbacktheta} how the classical theorem of Abel gives a relation between the tau function and the theta function, and in Theorem \ref{thm:tauthetamoresing} we make this precise in the case of algebraic theta divisors. In particular, when the curve is rational with a unique unibranch singularity, both the tau function and the theta function are polynomial, as shown in Proposition \ref{lemma:shapetau}, and each one can be recovered from the other: see Theorem \ref{thm:tauthetaonesing} and Corollary \ref{cor:thetafromtau}. In particular, this gives an explicit way to compute the algebraic theta function of such a curve and it is also used in the proof of Theorem \ref{theoremDeg} on the degrees of the algebraic theta divisors. It is interesting to note that special cases of the algebraic theta functions we are studying here, have appeared in a number of other works in the literature on PDE, in particular in \cite{BEL}, which treats the case of the so-called $\langle n,s \rangle$-curves. We vastly generalize their result to all rational Gorenstein curves with unibranch singularities.
	
	Our paper is organized as follows: in Section 2, we set up some notation and terminology about singular curves and generalized Jacobians, and we review the construction of the Sato Grassmannian~\cite{Sato}, which we use to describe the tau functions of the KP hierarchy.  In Section 3, we develop  the connection between the theta functions and tau functions in detail.
	In Section 4 we characterize the irreducible singular curves that yield algebraic theta divisors. We will present a number of examples of such algebraic theta divisors in higher dimensions supplementing Eiesland's list in \cite{Eies08}.  The corresponding KP tau functions are also computed.  We next consider the degrees of algebraic theta divisors in Section 5, and we actually prove a stronger version of Theorem \ref{theoremDeg} which describes also the leading term of the algebraic theta function. The proof works by degenerating to a curve with monomial singularities. Finally, in Section 7 we will present a few examples coming from reducible curves.  Here the situation is more complicated because the theta divisor will also be reducible
	(see Example~\ref{GenusFourNonAlgebraic} and Example~\ref{RedEx}).
	We cannot claim to have anything like a complete classification of the algebraic theta divisors for this reason. 
	
	Finally, we should point out that our point of view here is relatively naive in that we will not make use of any of the more sophisticated constructions of a compactified Jacobian~\cite{Kass,Cap}, and we just work on the generalized Jacobian.
	\medskip
	
	\textbf{Acknowledgements:} We thank Bernd Sturmfels for bringing us together and for his continuous encouragement of this project. The second author is supported by Turkish Scientific and Technological Research Council (TÜB\.{I}TAK) -- TÜB\.{I}TAK 2236, project number 1119B362000396.

	\section{Notation and Preliminaries}\label{sec:NotPre}
	
	\subsection{Singular Curves}
	
	We will work entirely over $\CC$.  
	We begin by recalling some facts about singular curves, their dualizing sheaves, their canonical images, and 
	the structure of curve singularities. Our main references will be  \cite{Rosenlicht52,Rosenlicht54,Serre,KleimanMartins}. 
	
	Let $C$ be a reduced and connected curve.  We will write $\nu : \widetilde{C} \to C$ for 
	the normalization. For each point $P\in C$, the integral closure of $\mathcal{O}_{C,P}$ is exactly $(\nu_*\mathcal{O}_{\widetilde{C}})_P$. This is a finite extension and the \emph{delta-invariant} of the point is defined as $\delta_P := \dim_{\mathbb{C}} (\nu_*\mathcal{O}_{\widetilde{C}} )_P/\mathcal{O}_{C,P}$. It is positive if and only if $P$ is a singular point. If the curve $C$ is projective and irreducible of arithmetic genus $g$, this can be computed as $g=g(\widetilde{C})+\sum_{P\in C} \delta_P$. Furthermore, if $\mathfrak{c}_P$ is the conductor of $\mathcal{O}_{C,P}$ in $(\nu_*\mathcal{O}_{\widetilde{C}})_P$ and if we set $d_P:= \dim_{\mathbb{C}} (\nu_*\mathcal{O}_{\widetilde{C}})_P/\mathfrak{c}_P$, then for each singular point $P$, the inequalities  $\delta_P+1\leq d_P \leq 2\delta_{P}$ hold. A singularity is called \emph{Gorenstein} if $d_P=2\delta_P$ and a curve is Gorenstein if so are all its singularities. For example, any curve in a smooth surface is Gorenstein.
	
	The dualizing sheaf $\omega_C$ was mentioned briefly above: its sections on an open subset $U\subseteq C$ can be viewed as meromorphic differentials $\omega$ on $\nu^{-1}(U)$ with the property that 
	\begin{equation}
	\label{RosenlichtDiff}
	\sum_{Q \in \nu^{-1}(P)} {\rm Res}_Q(\nu^*(f)\cdot \omega) = 0 \qquad  \text{ for all } P\in U \text{ and all } f\in {\cal O}_{C,P}.
	\end{equation}
	A curve is Gorenstein if and only if $\omega_C$ is an invertible sheaf on $C$. The global sections $H^0(C,\omega_C)$ are called Rosenlicht or canonical differentials, and the above description implies that they have  no poles at points
	on $\widetilde{C}$ that map to smooth points of $C$.  If the curve $C$ is projective and irreducible of arithmetic genus $g$ then $h^0(C,\omega_C)=g$, while $h^0(C,\omega_C) \geq g$ in general.

	Rosenlicht showed that when $C$ is irreducible, the global sections of $\omega_C$ define
	a base-point-free linear series on $\widetilde{C}$, giving a \emph{canonical mapping} 
	$$\phi : \widetilde{C} \to \mathbb{P}^{g-1}$$
	where $g$ is the arithmetic genus of $C$.  If there is no two-to-one mapping from 
	$C$ to $\mathbb{P}^1$, then $C$ is said to be \emph{non-hyperelliptic}.   The curve $\phi(\widetilde{C})$
	is called the canonical image, or canonical model.  When $C$ is non-hyperelliptic and 
	Gorenstein, that is, when $\omega_C$ is an invertible sheaf, 
	$C$ and $\phi(\widetilde{C})$ are isomorphic.  The curve $\phi(\widetilde{C})$ has degree $2g - 2$ in $\mathbb{P}^{g-1}$, 
	and since it also has arithmetic genus $g$, it has the largest possible genus for a curve of its degree.  (Recall, we are still assuming $C$ is irreducible.)   Conversely, 
	Rosenlicht showed (Corollary to Theorem 16 and Theorem 17 of \cite{Rosenlicht52}) that if $C$ is any integral curve of degree $2g - 2$ in 
	$\mathbb{P}^{g-1}$ having arithmetic genus $g$, then $C$ is non-hyperelliptic and Gorenstein.  In other words, these curves are equal to their own canonical models if a suitable basis for ${\rm H}^0(C,\omega_C)$ is used to construct the canonical mapping.  
	
	\begin{example}
		\label{Monomial456}
		Consider the parametrized rational curve $C$ which is the  
		image of 
		\begin{align*}
		\nu : \mathbb{P}^1 &\longrightarrow \mathbb{P}^3\\
		(u:t) &\longmapsto (u^6: t^4u^2: t^5u : t^6).
		\end{align*}
		$C$ is a curve of degree $6$ with exactly one singular point $Q = (1:0:0:0)$.  The mapping $\nu$ is exactly the normalization of $C$ and we look at $u$ and $t$ as local coordinates on the charts $\{ t\ne 0 \}$ and $\{ u\ne 0 \}$ of $\mathbb{P}^1$ respectively.  
		The ring ${\cal O}_{C,P}$ is the localization of $\CC[t^4, t^5, t^6]$ at $t = 0$.  
		It is easy to check that every 
		power of $t$ other than $t, t^2, t^3, t^7$ is contained in ${\cal O}_{C,P}$. On the other hand $(\nu_*\mathcal{O}_{\mathbb{P}^1})_P$ 
		is equal to the localization of $\CC[t]$ at $t = 0$.  Hence $\delta_P= 4$.  The conductor here is ${\mathfrak c}_P = \langle t^8\rangle$
		and hence $d_P =  8$.  Since we have $d_P = 2\delta_P$, this singular point is Gorenstein.  The arithmetic genus of $C$ is $g = g(\mathbb{P}^1) + \delta_P = 4$.  Since $C$ has degree 
		$6 = 2\cdot 4 - 2$ in $\mathbb{P}^3$, this is an example of the curves isomorphic to their own canonical models discussed above.  This can be 
		checked by computing a basis for ${\rm H}^0(C,\omega_C)$ as $\left( du,u\,du,u^2du,u^6du \right)$.  $\triangle$
	\end{example}
	
	A singularity $P \in C$ whose set-theoretic preimage $\nu^{-1}(P)$ consists of a single point $\widetilde{P}\in \widetilde{C}$, such as in Example \ref{Monomial456}, is called \emph{unibranch} or \emph{analytically irreducible}.
	
	For unibranch singular points, the Gorenstein property is also equivalent to a \emph{symmetry property} of the so-called \emph{value semigroup}. This is the semigroup $\mathbb{H}_P\subseteq \mathbb{N}$ generated by the orders of vanishing of the elements $\nu^*f$ for $f\in \mathcal{O}_{C,P}$ at the point $\widetilde{P}\in \widetilde{C}$. The set of \emph{gaps} is  $\mathbb{W}_P:=\mathbb{N}\setminus \mathbb{H}_P$. One can observe that a unibranch singular point $P$ is Gorenstein if and only if 
	the value semigroup $\mathbb{H}_P$ satisfies 
	\begin{equation}
	\label{symmetry}
	w \notin \mathbb{H}_P \Leftrightarrow 2\delta_p - 1 - w \in \mathbb{H}_P.
	\end{equation}
	The original reference for this is \cite{Kunz}.  For instance, in Example~\ref{Monomial456}, the semigroup in question is 
	$\mathbb{H}_P = \langle 4,5,6\rangle  = \{0,4,5,6,8,\cdots\} \subset \mathbb{N}.$ 
	The set of gaps is $\mathbb{W}_P=\{1,2,3,7\}$.  Note
	that $w$ is a gap if and only if $7 - w \in \mathbb{H}_P$, so the symmetry property 
	\eqref{symmetry} holds.
	
	\subsection{Weierstrass gap sequences}\label{sec:weierstrassgap}
	
	Assume now that the curve $C$ is irreducible, Gorenstein and has only unibranch singularities. To any smooth point $P$ one associates a Weierstrass gap sequence $\mathbb{W}_P = \{ w_1 < w_2 < \dots < w_g \}$ in $\mathbb{N}$ as the set of natural numbers $k$ such that $h^0(C,\mathcal{O}_C(kP)) = h^0(C,\mathcal{O}_C((k-1)P))$, or, equivalently, $h^0(C,\omega_C(-kP)) = h^0(C,\omega_C(-(k-1)P))-1$. The sequence consists exactly of $g=g(C)$ elements and furthermore $1\leq w_1$ and $w_g\leq 2g-1$. To such a sequence we can associate a partition $\lambda = (\lambda_1,\dots,\lambda_g)$ with exactly $g$ parts by setting 
	\begin{equation*} \lambda := (w_g,w_{g-1},\dots,w_2,w_1) - (g-1,g-2,\dots,1,0), 
	\end{equation*}
	where the subtraction is performed componentwise.
	
	If instead $P$ is a singular point of $C$, we can associate to it a gap sequence as $\mathbb{W}_P = \mathbb{N}\setminus \mathbb{H}_P$, where $\mathbb{H}_P$ is the value semigroup of the singularity. There are exactly $\delta = \delta_P$ gaps $w_1<w_2<\dots<w_{\delta}$ and the symmetry condition \eqref{symmetry} implies $w_1 = 1$ and $w_\delta = 2\delta - 1$. We can also associate to this a partition $\lambda = (\lambda_1,\dots,\lambda_\delta)$ with exactly $\delta$ parts  by taking
	\begin{equation}\label{partition} \lambda := (w_\delta,w_{\delta-1},\dots,w_2,w_1) - (\delta-1,\delta-2,\dots,1,0). \end{equation} It was proven in \cite{BEL} that such a partition coming from a Gorenstein unibranch singularity is always symmetric and a subpartition of the triangular partition $(\delta,\delta-1,\dots,1)$. We note that these statements are proved in \cite{BEL} for the semigroups generated by two relatively prime integers $n,s$: $\mathbb{H}_{n,s} = \langle n, s\rangle$.  These curves are called $(n,s)$-curves in \cite{BEL}. But the arguments only make use of \eqref{symmetry}, hence all the statements above follow for the semigroup of any Gorenstein singular point, not just planar singular points.  
	
	\begin{example}
		Consider the curve $C$ of Example \ref{Monomial456}. It is easy to see that the Weierstrass partition associated to the singular point $P$ is $(4,1,1,1)$. 	On the other hand we can also compute the Weierstrass partition associated to the smooth point $P_0=(0:0:0:1)$ as $(4,1,1,1)$.
	\end{example}
	
	\subsection{Generalized Jacobians, Abel maps, and theta functions}
	
	Let now $C$ be projective and let $C_0=C\setminus \textrm{Sing} (C)$ be the smooth locus of $C$. The \emph{generalized Jacobian} of $C$, denoted $\textrm{Jac}(C)$, is defined as the quotient of the group of the Cartier divisors on $C_0$ of degree zero on each component, by the group of divisors of meromorphic functions on $C$. On the other hand, we have a map $
	{\rm H}_1(C_0,\mathbb{Z}) \rightarrow  {\rm H}^0(C,\omega_C)^*$
	induced by integrating the canonical differentials over 1-chains. The image of this map is a discrete subgroup of rank at most $2\dim {\rm H}^0(C,\omega_C)$ and we denote it $\Lambda_C$. The \emph{Abel map} is defined by choosing a nonsingular point $P_i$ on each irreducible component $C_i$ of $C$ as follows:
	\begin{equation}
	\label{AbelMap}
	\alpha : C_0\longrightarrow {\rm H}^0(C,\omega_C)^*/\Lambda_C, \quad P \longmapsto \left(\omega \mapsto \int_{P_i}^P\omega \right) \quad \text{ for } P\in C_i.
	\end{equation}
	By additivity, this can be extended to a map $\alpha^{(n)}\colon C_0^{(n)} \to H^0(C,\omega_C)^*/\Lambda_C$ from the symmetric product of $C_0$, that is, the set of effective divisors on $C_0$ and then Abel's Theorem asserts that this establishes an isomorphism $\operatorname{Jac}(C) \cong {\rm H}^0(C,\omega_C)^*/\Lambda_C$. 
	
	We now assume that $C$ is Gorenstein of genus $g$. Then $H^0(C,\omega_C)$ is a complex vector space of dimension $g$ and  $\Lambda_C$ is a sublattice of rank at most $2g$. We set 
	$$W_{g-1} = \alpha^{(g-1)}(C_0^{(g-1)}) \subseteq \operatorname{Jac}(C):$$ the Zariski closure $\overline{W}_{g-1}$ in $\operatorname{Jac}(C)$ is an irreducible divisor in $\operatorname{Jac}(C)$, that we call the \emph{theta divisor}. The theta divisor corresponds to an analytic hypersurface in $H^0(C,\omega_C)^*$, which can be described by an analytic equation after choosing a basis of this space:
	\begin{equation}\label{eq:theta} 
	\theta(z_1,\dots,z_g) = 0. 
	\end{equation}
	This is called the \emph{theta function} of the curve $C$.
	
	\begin{remark}\label{rmk:princpol}
		The theta divisor $W_{g-1}$ in the Jacobian of a smooth curve defines a principal polarization on $\operatorname{Jac}(C)$.  This implies that if $x\in \operatorname{Jac}(C)$, then $W_{g-1}+x = W_{g-1}$ if and only if $x=0$ in $J(C)$.  A corresponding statement for the  $\overline{W_{g-1}}$ of a singular curve can be proven as for smooth curves, since the only ingredients needed are Abel's theorem and Riemann-Roch. In particular, this means that the corresponding analytic hypersurface in $H^0(C,\omega_C)^*$ is translation-invariant only with respect to the lattice $\Lambda_C$, and nothing more.
	\end{remark}

	\begin{remark}\label{rmk:sequencejacobians}
		We record here some facts on the structure of the generalized Jacobian, following \cite[Section 9.2]{BLR}.	There is a partial normalization $C'$ of $C$ obtained by identifying all points of $\widetilde{C}$ lying over the singular points of $C$ in such a way that the resulting curve has all singularities locally isomorphic to the union of coordinate axes in a certain affine space. There are natural maps $\widetilde{C} \longrightarrow C' \longrightarrow C$ which induce inclusions $H^0(\widetilde{C},\omega_{\widetilde{C}}) \subseteq H^0(C',\omega_{C'}) \subseteq H^0(C,\omega_C)$ and surjections $\Lambda_{C} \twoheadrightarrow \Lambda_{C'} \twoheadrightarrow \Lambda_{\widetilde{C}}$, so that we have surjective morphisms of algebraic groups
		\begin{equation}\label{eq:mapsjacobians}
		\operatorname{Jac}(C) \twoheadrightarrow \operatorname{Jac}(C') \twoheadrightarrow \operatorname{Jac}(\widetilde{C}).
		\end{equation}
		The kernel of the first map is a unipotent group, and the kernel of the second map is an algebraic torus, whereas $\operatorname{Jac}(\widetilde{C})$ is an abelian variety.  When $C$ is irreducible with only unibranch singularities, for instance, this means $C'$ and $\widetilde{C}$ are isomorphic. 
	\end{remark}
	
	\subsection{
		The Sato Grassmannian}  
	
	The KP hierarchy is a family of partial differential equations which generalizes the KP equation. All solutions of these equations can be described via the Sato Grassmannian~\cite{Sato}, which we recall now. First, the \emph{elementary Schur-Weierstrass polynomials} $\sigma_i({\bf x})$, are defined by the generating series
	\begin{equation}\label{eq:elementarySW} 
	\exp\left( \sum_{i=1}^{\infty} x_i t^i \right) = \sum_{i=0}^{\infty} \sigma_i({\bf x})t^i.
	\end{equation}
	They are polynomials in the infinitely many variables ${\bf x}=(x_1,x_2,x_3,\dots)$.  Then, to any partition $\lambda = (\lambda_1,\lambda_2,\dots,\lambda_m)$ with $\lambda_1\geq \lambda_2 \geq \dots \geq \lambda_m>0$, we can associate a \emph{Schur-Weierstrass polynomial} by
	\begin{equation}\label{eq:SW} 
	\sigma_{\lambda}({\bf x}) = \det \left(  \sigma_{\lambda_i+j-i}\right({\bf x}))_{1\leq i,j \leq m}. 
	\end{equation}
	If we give weight $i$ to $x_i$, then each $\sigma_{\lambda}$ is homogeneous of weight $|\lambda|:=\sum \lambda_i$. 
	These polynomials are related to the usual symmetric Schur polynomials as follows: take $n$ variables $u_1,\dots,u_n$ and consider the power-sums $p_i = u_1^{i}+\dots+u_n^i$, then
	\begin{equation}\label{eq:SWtoS} 
	\sigma_{\lambda}\left( p_1,\frac{1}{2}p_2,\frac{1}{3}p_3,\dots \right) = s_{\lambda}(u_1,\dots,u_n), 
	\end{equation}
	where $s_{\lambda}$ is the usual symmetric Schur polynomial. 
	
	Now we discuss the Sato Grassmannian. Let $V=\mathbb{C}(\!(u)\!)$ be the field of formal Laurent series with complex coefficients.
	Consider the natural projection map $\, \pi\colon V \to \mathbb{C}[u^{-1}] \,$
	onto the polynomial ring in $u^{-1}$. We regard $V$ and $ \mathbb{C}[u^{-1}] $ as
	$\mathbb{C}$-vector spaces, with Laurent monomials $u^i$ serving as elements of a basis.
	Points in the Sato Grassmannian {\rm SGM} correspond to
	$\mathbb{C}$-subspaces $U\subset V$ such that
	\begin{equation}
	\label{eq:kernelcokernel}
	\dim \operatorname{Ker} \pi_{|U} \,\,= \,\, \dim \operatorname{Coker } \pi_{|U} ,
	\end{equation}	
	and this common dimension is finite.
	We can represent $U \in {\rm SGM}$ via a doubly infinite matrix as follows. For any basis $(f_1,f_2,f_3,\dots)$ of $U$,
	the $j$th basis vector is a Laurent series,
	\[
	f_j(u) \,\,=\,\, \sum_{i=-\infty}^{+\infty} \xi_{i ,j}u^{i+1}.
	\]
	Then $U$ is the column span of the infinite matrix $\xi = (\xi_{i,j})$ whose rows are indexed from top to bottom by $\ZZ$ and whose
	columns are indexed from right to left by $\NN$. The $i$-th row of $\xi$ corresponds to the coefficients of $u^{i+1}$.
	Sato proved that a subspace $U$ of $V$ satisfies (\ref{eq:kernelcokernel}) 
	if and only if there is a basis, called a \emph{frame} of $U$, whose corresponding matrix has the shape
	
	\begin{equation}
	\label{eq:xishape}
	\xi \,\,=\,\, \begin{small} \begin{pmatrix} \ddots  & \vdots & \vdots & \vdots & \vdots & \cdots & \vdots \\
	\cdots & \mathbf{1} & 0 & 0 & 0 & \cdots & 0 \\ 
	\cdots & * & \mathbf{1} & 0 & 0 & \cdots & 0 \\ 
	\cdots & * & * & \xi_{-\ell,\ell} & \xi_{-\ell,\ell-1} & \cdots & \xi_{-\ell,1} \\ 
	\cdots & * & * & \xi_{-\ell+1,\ell} & \xi_{-\ell+1,\ell-1} & \cdots & \xi_{-\ell+1,1} \\ 
	{}  & \vdots & \vdots & \vdots & \vdots & \cdots & \vdots \\
	\cdots & * & * & \xi_{-1,\ell} & \xi_{-1,\ell-1} & \cdots & \xi_{-1,1} \\
	\cdots & * & * & \xi_{0,\ell} & \xi_{0,\ell-1} & \cdots & \xi_{0,1} \\ 
	\cdots & * & * & \xi_{1,\ell} & \xi_{1,\ell-1} & \cdots & \xi_{1,1} \\
	{}  & \vdots & \vdots & \vdots & \vdots & \cdots & \vdots 
	\end{pmatrix}. \end{small}
	\end{equation}
	This matrix is infinite vertically, infinite on the left and, most importantly, it is eventually lower triangular with $1$'s on the diagonal, at the $(-n,n)$ positions. 
	The space $U$ is described by the positive integer $\ell$ and the submatrix with
	$\ell$ linearly independent columns whose upper left entry is $\xi_{-\ell,\ell}$.
	This description implies that a subspace
	$U$ of $V$ satisfies (\ref{eq:kernelcokernel}) if and only~if
	\begin{equation}
	\label{eq:kernelcokernel2}
	\text{there exists $\ell \in \mathbb{N}$ \, such that} \quad
	\dim U\cap V_{n} \,\,=\,\, n+1 \quad \text{ for all } n\geq \ell,
	\end{equation}
	where $V_n = u^{-n}\mathbb{C}[\![u ]\!]$ denotes the space of Laurent series with a pole of order at most $n$.
	
	The \emph{Pl\"ucker coordinates} on SGM are computed as minors
	$\xi_\lambda$ of the matrix $\xi$.
	Think of a partition $\lambda$ as a
	weakly decreasing sequence of nonnegative integers
	that are eventually zero. 
	Setting $m_i = \lambda_i - i$ for $i \in \NN$,
	we obtain the  associated {\em Maya diagram}  $(m_1,m_2,m_3,\dots)$. This is a vector of strictly decreasing integers $m_1 > m_2 >  \dots$ such that $m_i = -i$ for large enough $i$.
	Partitions and Maya diagrams are in natural bijection.
	Given any partition $\lambda$, we consider the
	matrix $(\xi_{m_i,j})_{i,j\geq 1}$ whose row indices $m_1,m_2,m_3,\dots$ are the entries
	in the Maya diagram of $\lambda$.
	Thanks to the shape of the matrix $\xi$, it makes sense to take the determinant
	\begin{equation}
	\label{eq:SGMpara}
	\xi_{\lambda} \,\,:= \,\, \det(\xi_{m_i,j}). 
	\end{equation}
	This Pl\"ucker coordinate is a scalar in $\mathbb{C}$ that can be
	computed as a maximal minor of the finite matrix to the lower right of $\xi_{-\ell,\ell}$ in (\ref{eq:xishape}). These Pl\"ucker coordinates can be used to define a tau function
	\begin{equation}
	\label{eq:tau} 
	\tau({\bf x}) = \sum_{\lambda} \xi_{\lambda} \cdot \sigma_{\lambda}({\bf x}) 
	\end{equation} 
	where $\sigma_{\lambda}({\bf x})$ is the Schur-Weiestrass polynomial corresponding to the partition $\lambda$. A fundamental result of Sato and Segal-Wilson~\cite{Sato, SegWil} is that the  function $\tau({\bf x})$  is a solution to the KP-hierarchy and moreover every such solution arises in this way. In particular, by setting
	\[ f(x,y,t) := 2 \frac{\partial^2}{\partial x^2} \log \tau(x,y,t,0,0,\dots) \]
	we obtain solutions to the KP equation.

	\section{Theta functions and tau functions}\label{sec:thetatau}

	In this section, we connect the theta function for an algebraic curve with a tau function for the KP hierarchy. The key point is Krichever's construction~\cite{Kri}, which associates points in the Sato Grassmannian to algebraic curves. Let $C$ be a reduced and irreducible Gorenstein projective curve of arithmetic genus $g$, and let $C_0$ be its smooth locus. We consider a line bundle of degree $g-1$ of the form $L=\omega_C(-D)$, where $D$ is a Cartier divisor of degree $g-1$ supported on $C_0$. 
	We fix a smooth point $P_0\in C$ and we consider the space
	\[ H^0(C,L(\infty P_0)) = \bigcup_{n\geq 0} H^0(C,L(nP_0)) \]
	of differentials with poles described by $D$ and of arbitrary order at $P_0$. We also fix a local coordinate $u$ around $P_0$.  Then we can uniquely write any element of $H^0(C,L(\infty P_0))$ as $\omega = f(u)du$, where $f(u)$ is a Laurent series, and we can define the map
	\[  \iota \colon H^0(C,L(\infty P_0)) \longrightarrow V=\mathbb{C}(\!(u)\!), \qquad \omega = f(u)du \mapsto u^{1-m}f(u) \]
	where $m=\operatorname{mult}_{P_0} D$ is the multiplicity of $D$ at $P_0$. Then the following statement is known~\cite[Section 6]{SegWil}.
	
	\begin{proposition}
		The image $U=\iota(H^0(C,L(\infty P_0)))$ belongs to the Sato Grassmannian. 
	\end{proposition}
	
	
	\begin{remark}\label{rmk:framecurve}
		In particular, one can find a frame of $U= \iota(H^0(C,L(\infty P_0)))$ as $\iota(\omega'_1),\dots,\iota(\omega'_{g}),\iota(\omega'_{g+1}),\dots$ , where $\omega'_1,\dots,\omega'_g$ are a basis of $H^0(C,L(gP_0))$ and $\omega'_{g+i}$ has a pole of order $g+i-m$ at $P_0$ for all $i\geq 1$ and the other poles are bounded by  $D$.
	\end{remark}
	
	At this point we can make connections with the Abel maps considered before. More precisely, let's consider $n\geq 1$ and the non-symmetric Abel map with basepoint at $P_0$
	\[ a_n\colon {C}_0^n \longrightarrow \operatorname{Jac}(C), \qquad (P_1,\dots,P_n) \mapsto \int_{P_0}^{P_1} \boldsymbol{\omega} + \dots + \int_{P_0}^{P_n} \boldsymbol{\omega} \]
	where $\boldsymbol{\omega}=(\omega_1,\dots,\omega_g)^t$. The non-symmetric Abel map is the composition of the usual projection $\pi\colon C_0^n \longrightarrow C_0^{(n)}$ with the symmetric Abel map $\alpha\colon C_0^{(n)} \longrightarrow \operatorname{Jac}(C)$. Let us discuss briefly the coordinates in these spaces: if $u$ is a local coordinate around $P_0$, we have corresponding coordinates $(u_1,\dots,u_n)$ around $(P_0,\dots,P_0)\in C_0^n$. A system of coordinates around the point $nP_0 \in C_0^{(n)}$ is then given by the power sums $(x_1,\dots,x_n)$ where
	\[ x_i = \frac{1}{i}(u_1^i+\dots+u_n^i) \]
	and finally we have the coordinates $(z_1,\dots,z_g)$ on $\operatorname{Jac}(C)$.

	The divisor $W_{g-1}$ is the image $a_{g-1}(C_0^{g-1}) \subseteq \operatorname{Jac}(C)$. Moreover, we can consider also the image of the divisor of degree zero $D-(g-1)P_0$ in $\operatorname{Jac}(C)$ via the Abel map, which we denote by $\alpha(D-(g-1)P_0)$. Then we can consider the translate $W_{g-1}-\alpha(D-(g-1)P_0)$, and by Abel's theorem, the pullback of this divisor along the non-symmetric Abel map is given by  $a_n^*(W_{g-1}-\alpha(D-(g-1)P_0)) $
	\begin{align*} 
	&= \left\{ (P_1,\dots,P_n) \in {C_0}^n \,|\, h^0\Big(C,{\cal O}_C\Big(\sum_i P_i +D-nP_0\Big)\Big) \ne 0  \right\} \\
	& = \left\{ (P_1,\dots,P_n) \in {C_0}^n \,|\, h^0(C,\omega_C\Big(\Big(-D+nP_0-\sum_i P_i\Big)\Big) \ne 0  \right\}\\
	& = \left\{ (P_1,\dots,P_n) \in {C_0}^n \,|\, h^0\Big(C,L\Big(nP_0-\sum_i P_i\Big)\Big)) \ne 0  \right\}.
	\end{align*}
	Now we suppose $n\geq g$ and we focus our attention around the point $(P_0,\dots,P_0) \in C_n$. In particular, this belongs to $a_n^*(W_{g-1}-\alpha(D-(g-1)p))$ if and only if $h^0(C,L)>0$. Take a basis $\omega'_1,\dots,\omega'_n$ of $H^0(C,L(nP_0))$ as in Remark \ref{rmk:framecurve} and expand them as Laurent series around $P_0$ with the coordinate $u$:
	\begin{equation}\label{eq:expomega}
	\omega'_j(u) = \sum_{i=-n}^{\infty} \xi_{ij} u^{i+m} du.
	\end{equation}
	Around the point $(P_0,\dots,P_0)$ we have local coordinates $(u_1,\dots,u_n)$ in $C_0^n$. With this notation, we have the following
	
	\begin{lemma}
		Around the point $(P_0,\dots,P_0)$ and with coordinates $(u_1,\dots,u_n)$, the divisor $a_n^*(W_{g-1}-\alpha(D-(g-1)p))$ has equation
		\begin{equation}\label{eq:deglocusomega} 
		\frac{(u_1\dots  u_n)^{n-m}}{\Delta(u_1,\dots,u_n)}\cdot\det \begin{pmatrix} \omega'_n(u_1) & \omega'_{n-1}(u_1) & \dots & \omega'_2(u_1) & \omega'_1(u_1) \\
		\omega'_n(u_2) & \omega'_{n-1}(u_2) & \dots & \omega'_2(u_2) & \omega'_1(u_2)\\
		\vdots & \vdots & \dots & \vdots & \vdots \\
		\omega'_n(u_n) & \omega'_{n-1}(u_n) & \dots & \omega'_2(u_n) & \omega'_1(u_n)
		\end{pmatrix} = 0,
		\end{equation}
		where $\Delta$ is the Vandermonde determinant $\Delta(u_1,\dots,u_n) = \prod_{1\le i < j\le n} (u_i-u_j)$. 
	\end{lemma}
	\begin{proof}
		We focus on a small neighborhood $\mathcal{U}$ of $P_0$ with local coordinate $u$. Then we note that $u^{n-m}\omega_j' = f_j(u)du$, where $f_j(u)$ is a holomorphic function on $\mathcal{U}$. We have the coordinates $(u_1,\dots,u_n)$ on the cartesian product $\mathcal{U}^n$ and the coordinates $(x_1,\dots,x_n)$ on  the symmetric product $\mathcal{U}^{(n)}$, with the projection map $\pi\colon \mathcal{U}^n\to \mathcal{U}^{(n)}$ given by $x_i = \frac{1}{i}(u_1^i+\dots+u_n^i)$. Now let us consider the Abel-like map
		\[ \alpha'\colon \mathcal{U}^{(n)} \longrightarrow \mathbb{C}^n, \qquad P_1+\dots+P_n \mapsto \int_0^{P_1} \mathbf{f}(u)du + \dots + \int_0^{P_n} \mathbf{f}(u)du \]
		where $\mathbf{f} = (f_1(u),\dots,f_n(u))^t$. Then the computations in \cite[Section IV.1]{ACGH} show that $Z=\{ P_1+\dots+P_n \in \mathcal{U}^{(n)} \,|\, h^0(C,L(nP_0-P_1-\dots-P_n)) \}$ is exactly the ramification divisor of the map $\alpha'$, meaning that it coincides, as a scheme, with the locus where the differential $d\alpha'$ is not of maximal rank. The previous computations show that $a_n^*(W_{g-1}-\alpha(D-(g-1)P_0)) = \pi^* Z$. Now consider the non-symmetric Abel map
		\[ a' = (\alpha'\circ \pi)\colon \mathcal{U}^n \longrightarrow \mathbb{C}^n, \qquad (u_1,\dots,u_n) \mapsto \int_0^{u_1}\mathbf{f}(u)du+\dots+\int_0^{u_n} \mathbf{f}(u)du.  \]
		Since $da' = d\alpha'\circ d\pi$, the ramification divisor of $a'$ is given by $\pi^*Z$ plus the ramification divisor of $\pi$. The divisor $\pi^*Z$ is the one we are looking for, and it is easy to see that the ramification divisor of $\pi$ is given by the Vandermonde determinant. To conclude, we observe that the differential of $a'$ is given exactly by the Brill-Noether matrix $M = \left( f_i(u_j) \right)$, and then the previous discussion shows that $\pi^*Z$ has a local equation $(\det M)\cdot (\det d\pi)^{-1} = 0$.
	\end{proof}

	We can rewrite this matrix using the expansion \eqref{eq:expomega} of the $\omega'_j$. We denote by $\xi^{(n)}$ the $\infty \times n$ matrix $\xi^{(n)} = (\xi_{ij})_{i\geq -n,j=1,\dots,n}$. This is the submatrix of the frame \eqref{eq:xishape} obtained by taking the first $n$ columns and all the rows from $-n$ onwards. Then we set $U^{(n)}$ to be the $n\times \infty$ matrix given by
	\begin{equation}
	U^{(n)} = \begin{pmatrix}
	u_1^{-n+m} & u_1^{-n+m+1} & \dots & u_1 & u_1^2 & \dots \\
	u_2^{-n+m} & u_2^{-n+m+1} & \dots & u_2 & u_2^2 & \dots \\
	\vdots & \vdots & \dots & \vdots & \vdots & \dots \\
	u_n^{-n+m} & u_n^{-n+m+1} & \dots & u_n & u_n^2 & \dots \\
	\end{pmatrix} = \left( u_i^{j} \right)_{i=1,\dots,n\,\, j\geq -n}.
	\end{equation}
	With this notation, we can rewrite the equation \eqref{eq:deglocusomega} simply as 
	\[ \frac{(u_1\dots u_n)^{n-m}}{\Delta(u_1,\dots,u_n)}\det(U^{(n)}\cdot \xi^{(n)}) = 0 \]
	and, using the Binet-Cauchy formula for the determinant, this becomes
	\begin{align*}
	\frac{(u_1\dots u_n)^{n-m}}{\Delta(u_1,\dots,u_n)}\det(U^{(n)}\cdot \xi^{(n)}) &=  \sum_{\lambda = (\lambda_1,\dots,\lambda_n)} s_{\lambda}(u_1,\dots,u_n)\cdot \xi_{\lambda}, 
	\end{align*}
	where the sum is over all partitions with at most $n$ parts, the $s_{\lambda}$ is the symmetric Schur polynomial and $\xi_{\lambda}$ is the Pl\"ucker coordinate of $\xi$ indexed by $\lambda$. Observe that this last expression is precisely what we obtain from the $\tau$ function \eqref{eq:tau} under the substitution $x_i = \frac{1}{i}p_i$, where $p_i = p_i(u_1,\dots,u_n)$ is the $i$-th power symmetric polynomial in the $u_i$. Indeed, under this substitution, the Schur-Weierstrass polynomials $\sigma_{\lambda}({\bf x})$ become the symmetric Schur polynomials $s_{\lambda}$, and moreover these are identically zero whenever $\lambda$ has more than $n$ parts, which is the number of variables. In summary, we have the following proposition.
	
	\begin{proposition}\label{prop:pullbacktheta}
		Around the point $(P_0,\dots,P_0)\in C_0^n$ the divisor
		$a_n^*(W_{g-1}-\alpha(D-(g-1)P_0))$ has local equation
		\[  \tau\left(p_1, \frac{1}{2}p_2, \frac{1}{3}p_3, \dots \right) = 0
		, \text{ where } p_i(u_1,\dots,u_n) = u_1^i+\dots+u_n^i. \]
	\end{proposition}
	
	We can use this to relate the tau function of the frame $\iota(H^0(C,L(\infty P_0)))$ with the theta function, i.e. the equation defining the divisor $W_{g-1}$. More precisely, suppose that the theta divisor $W_{g-1}$ is described by the theta function $\theta(z_1,\dots,z_g)$ around the point $(0,0,\dots,0)$, and suppose that  $\alpha(D-(g-1)P_0)$ can be represented with coordinates $\mathbf{b}=(b_1,\dots,b_g)^t$, so that the divisor $W_{g-1}-\alpha(D-(g-1)P_0)$ is described by the function $\theta(\mathbf{z}+\mathbf{b})$. Now we write the Abel map in coordinates: if $\omega_j=\left(\sum_{i=0}^{\infty} a_{ij}u^i \right)du$ are local expansions of the differentials $\omega_1,\dots,\omega_g$, then the coordinates of the Abel map are given by
	\[ z_j = \sum_{h=1}^n \int_{0}^{u_h} \omega_j = \sum_{h=1}^n  \sum_{i=0}^{\infty} a_{ij} \frac{1}{i+1}u_h^{i+1} = \sum_{i=0}^{\infty} a_{ij} \frac{1}{i+1}p_{i+1}, \quad \text{where } p_i=u_1^i+\dots+u_n^i.  \]
	Hence, if we define the  $g\times \infty$ matrix $A=(a_{ji})$, the pullback $a_n^*\left( W_{g-1}-\alpha(D-(g-1)P_0) \right)$ has local equation
	\begin{equation}\label{eq:equationtheta} 
	\theta(A\mathbf{p}+\mathbf{b}) = 0, \quad \text{where } \mathbf{p} = \left(p_1,\frac{1}{2}p_2,\dots\right)^t, \text{ with } p_i = u_1^i+\dots+u_n^i.
	\end{equation}
	However, another local equation is given by the one in Proposition \ref{prop:pullbacktheta}. We can combine these descriptions to obtain the following:
	\begin{theorem}\label{thm:tauthetageneral}
		Let $\tau(\mathbf{x})$ be the tau function corresponding to $H^0(C,L(\infty P_0))$. Then in the infinitely many variables $\mathbf{x} = (x_1,x_2,x_3,\dots)$, there is an identity of formal power series
		\[ \tau(\mathbf{x}) = \exp\left( C(\mathbf{x}) \right)\cdot \theta(A\mathbf{x}+\mathbf{b}) \]
		for a certain $C(\mathbf{x})$. 
	\end{theorem}
	\begin{proof}
		Take a small neighborhood $\mathcal{U}$ of $P_0$ with a local coordinate $u$. For every $n\geq g$, consider the substitution $x_i = \frac{1}{i}p_i$ and set $\tau_n(u_1,\dots,u_n) = \tau\left(p_1,\frac{1}{2}p_2,\dots\right)$ and $\theta_n(u_1,\dots,u_n) = \theta(A\mathbf{p}+\mathbf{b})$. Then Proposition \ref{prop:pullbacktheta} and \eqref{eq:equationtheta}, show that both $\tau_n$ and $\theta_n$ are local equations for the same divisor. Thus, they differ by a function that is everywhere nonzero on $\mathcal{U}^n$, meaning that 
		\[ \tau_n(u_1,\dots,u_n) = \exp\left( C_n(u_1,\dots,u_n) \right) \cdot \theta_n(u_1,\dots,u_n) \]
		for a certain function $C_n(u_1,\dots,u_n)$. Since $\tau_n$ and $\theta_n$ are symmetric, it must be that $C_n$ is symmetric as well, and since the power sums are a basis of symmetric functions, we can write $C_n = c_{n0}+\sum_{i=1}^\infty c_{ni} \frac{1}{i}p_i$ for certain coefficients $c_{ni}$. At this point, we observe that we can assume that $F_n(u_1,\dots,u_n) = F_{n+1}(u_1,\dots,u_n,0)$ since the same holds for $\tau_n$ and $\theta_n$. Hence we see that $c_{ni}=c_{gi}$ for all $n\geq g$. If we now define $C(\mathbf{x}) = c_{g0}+\sum_{i=1}^{\infty} c_{gi}x_i$ then the previous discussion together with Lemma \ref{lemma:equality} proves what we want.
	\end{proof}	
	
	\begin{lemma}\label{lemma:equality}
		Let $F(\mathbf{x})$ be a power series in the infinitely many variables $\mathbf{x}=(x_1,x_2,\dots)$. Then $F(\mathbf{x})=0$ if and only if for all $n$ large enough it holds that $F\left(p_1,\frac{1}{2}p_2,\frac{1}{3}p_3,\dots\right)=0$ where $p_i=u_1^i+\dots+u_n^i$.
	\end{lemma}	
	\begin{proof}
		The ring of power series is graded according to the weights $\deg x_i = i$. Suppose $F(\mathbf{x})\ne 0$ and let $F_m(\mathbf{x})$ be the term of smallest weight $m$. This is a polynomial in the finite number of variables $x_1,\dots,x_m$. Since the substitution $x_i=\frac{1}{i}p_i$ preserves the weight, the term of smallest degree in $F\left(p_1,\frac{1}{2}p_2,\frac{1}{3}p_3,\dots\right)$ is $F_m\left(p_1,\frac{1}{2}p_2,\frac{1}{3}p_3,\dots \right)$, hence it must be that $F_m\left(p_1,\frac{1}{2}p_2,\frac{1}{3}p_3,\dots \right)=0$ for all $n$  large enough. But taking $n\geq m$ we see that this implies $F_m(\mathbf{x})=0$, which is absurd. 
	\end{proof}


	\section{Algebraic theta divisors}
	
	\subsection{The Characterization}
	
	In this subsection we will provide a characterization of the curves yielding algebraic theta divisors. We start with one example where the divisor is actually algebraic and one where it is not.
	
	\begin{example}
		\label{ThetaDivisorMonomial456}
		Consider the curve $C$ of arithmetic genus $g = 4$ from Example~\ref{Monomial456}. We saw already that a basis of differentials is given by
		$\left( du, u\, du, u^2du,  u^6du\right)$. Since $C$ has a unique singularity, we see that $C_0\cong \mathbb{C}$, so that the homology $H_1(C_0,\mathbb{Z})$ is trivial. Hence $\operatorname{Jac}(C) \cong H^0(C,\omega_C)$ is simply an affine space so that the theta divisor is an algebraic hypersurface. To compute an equation, we use the parametrization:
		\begin{align*}
		z_1 &= \int_0^{u_1} du + \int_0^{u_2} du + \int_0^{u_3} du &
		z_2 &= \int_0^{u_1} u\,du + \int_0^{u_2} u\,du + \int_0^{u_3} u\,du\\
		z_3 &= \int_0^{u_1} u^2du + \int_0^{u_2} u^2du + \int_0^{u_3} u^2du &
		z_4 &= \int_0^{u_1} u^6du + \int_0^{u_2} u^6du + \int_0^{u_3} u^6du.
		\end{align*}
		We can carry out the integration and obtain that $W_{g-1} = W_3$ is the image of the map
		\begin{equation*}
		(u_1,u_2,u_3) \mapsto \left(p_1, \frac{1}{2}p_2, \frac{1}{3}p_3, \frac{1}{7}p_7 \right),
		\end{equation*}
		where the $p_i$ are the usual power sum polynomials.
		Then one can compute explicitly one equation of the theta divisor as
		\[ \theta(z_1,z_2,z_3,z_4) = z_1^7+21z_1^4z_3-84z_1^3z_2^2+252z_1 z_3^2+252 z_2^2 z_3-252z_4. \quad \quad \triangle\]
	\end{example}
	
	On the other hand, here is an example where the theta divisor is not algebraic.
	
	\begin{example}
		\label{nonunibranch}
		Consider the rational quartic plane curve $C$ of arithmetic genus $g = 3$ given as the image of the map 
		\begin{align*}
		\label{rattriple}
		\phi : \mathbb{P}^1 &\longrightarrow \mathbb{P}^2\\
		(u:t) &\longmapsto (t(t^2 - u^2)(t + u) : t^2(t^2 - u^2): u^4).
		\end{align*}
		The curve is rational by construction, but it has a triple point with distinct tangents at $(0:0:1)$, since 
		$\phi((1:-1)) = \phi((1:0)) = \phi((1:1)) = (0:0:1)$. 
		We can find a basis of canonical differentials as 
		\begin{equation*}
		\omega_1 = \frac{1}{1+u} \, du, \,\,\,
		\omega_2 = -\frac{1}{1 - u}\, du, \,\,\,
		\omega_3 = \left( \frac{1}{(1 + u)^2}+ \frac{4}{u^2} +\frac{1}{(1 - u)^2}\right)\ du.
		\end{equation*}
		Integrating these differentials with base point at $u=0$, we see that the hypersurface corresponding to the theta divisor in $H^0(C,\omega_C)$ can be parametrized as:
		\begin{small}
			\begin{align*}
			z_1 &= \int_0^{u_1} \omega_1 + \int_0^{u_2} \omega_1 = \log(1+u_1) +  \log(1 + u_2), \\
			z_2 &=  \int_0^{u_1} \omega_2 + \int_0^{u_2} \omega_2 = \log(1-u_1) + \log(1-u_2), \\
			z_3 &= \int_0^{u_1} \omega_3 + \int_0^{u_2} \omega_3 = -\frac{1}{1+ u_1} - \frac{4}{u_1}  + \frac{1}{1 - u_1}  -\frac{1}{1+ u_2} - \frac{4}{u_2}  + \frac{1}{1 - u_2}+8.
			\end{align*}
		\end{small}
		Then one can compute that $z_3$ is a rational function in $e^{z_1} = (1+u_1)(1+u_2)$ and $e^{z_2} = (1-u_1)(1-u_2)$ and hence
		the theta divisor has an equation of the form 
		$$A(e^{z_1},e^{z_2})z_3 -B(e^{z_1}, e^{z_2}) = 0,$$ where
		$A,B$ are polynomials of two variables.  The exact expressions for $A,B$ can also be derived, but
		it is actually just the form that is more important for us, since such a hypersurface cannot be algebraic, because it is periodic with respect to a discrete lattice of rank 2 in $\CC^3$.  $\triangle$
	\end{example}

	From the patterns seen in these two examples, we can completely characterize when the theta divisor is algebraic.
	
	\begin{proposition}
		\label{AlgCharac}
		Let $C$ be an irreducible and reduced projective Gorenstein curve of arithmetic genus $g$.  The theta divisor of $C$  is algebraic if and only if the curve is rational and all singular points are unibranch.
	\end{proposition}
	\begin{proof}
		First we observe that the theta divisor is algebraic if and only if $\Lambda_C=0$, i.e. $\operatorname{Jac}(C)=H^0(C,\omega_C)^{*}$. Indeed, if this is the case, then the theta divisor is a Zariski-closed subset in the vector space $H^0(C,\omega_C)^*$, hence an algebraic hypersurface. Conversely, suppose that $\Lambda_C\neq 0$. Then the theta divisor cannot be algebraic, otherwise we would have an algebraic hypersurface which is translation-invariant with respect to the discrete lattice $\Lambda_C$ and nothing else (see Remark \ref{rmk:princpol}). But this is impossible, for example because the hypersurface will have infinitely many discrete intersection points with a line.
		
		This being settled, assume first that the curve is rational with all unibranch singularities. Then the normalization of $C$ is $\mathbb{P}^1$ and  \eqref{SatoThetaDivisorMonomial456} shows that all the canonical differentials in $H^0(C,\omega_C)$ are meromorphic differentials on $\mathbb{P}^1$ with no residues at the points corresponding to the singularities of $C$ (these were classically called \emph{differentials of the second kind}). Since $H^1(C_0,\mathbb{Z})$ is generated by small cycles around these points, all canonical differentials integrate to zero on these cycles by the residue formula. Thus $\Lambda_C = 0$ and we are done. Alternatively, one can observe that if the residues of a rational differential $\omega$ on $\mathbb{P}^1$ are all zero, then the expression $\int_{P_0}^{u}\omega_C$ is a rational function of $u$. Indeed, there are no logarithms coming from the residues.
		
		Conversely, suppose that the theta divisor is algebraic. Then $\Lambda_C = 0$, and since this surjects to $\Lambda_{\widetilde{C}} \cong H^1(\widetilde{C},\mathbb{Z})$, this is zero as well. But then $\widetilde{C}$ is a smooth curve of genus zero, meaning $\widetilde{C}\cong \mathbb{P}^1$. Now, since $\Lambda_C \to \Lambda_{C'}$ is also surjective, we see that $\Lambda_{C'}=0$ as well, so that $\operatorname{Jac}(C')\cong H^0(C',\omega_C')$ is a unipotent group, but the discussion in Remark \ref{rmk:sequencejacobians} shows that $\operatorname{Jac}(C')$ is an algebraic torus. Hence, the only possibility is that $H^0(C',\omega_{C'})=0$, so that $C' \cong \mathbb{P}^1$ and the map $\mathbb{P}^1 \to C'$ is an isomorphism. But then by the construction of $C'$ it follows that $C$ has only unibranch singularities.
		
		This second step has another, more concrete proof as follows: From the discussion in 
		\S 2, we may identify $C$ and its canonical image in $\mathbb{P}^{g - 1}$.  
		Suppose on the contrary that there exists a
		singular point $p$ of $C$ such that $k = |\nu^{-1}(p)| \ge 2$ 
		on $\widetilde{C}$.   
		
		We claim first that there must be abelian differentials $\omega \in {\rm H}^0(C,\omega_C)$ such that 
		$\omega$ has nonzero residues at points in $\nu^{-1}(p)$.  For notational convenience, we will only consider the 
		case where $C$ has exactly \emph{one} such singular point and no others, 
		but the argument will generalize to other cases.  We can choose an affine
		coordinate $t$ on $\mathbb{P}^1$ such that $\nu^{-1}(p) = \{a_1,\ldots,a_k \}$ does not contain the point $t = \infty$.
		Hence $\omega$ cannot have a pole at $\infty$ since $\infty$ maps to a smooth point of $C$ by assumption.  
		Then $\omega$ is a rational differential on $\mathbb{P}^1$ of the form
		$$\omega = \frac{g(t)}{h(t)}\ dt$$
		where $\deg(h(t)) = 2g$ and $\deg(g(t)) \le \deg(h(t)) - 2 = 2g - 2$.  Since the only poles 
		of $\omega$ are $t = a_1,\ldots,a_k$,
		the denominator must factor as $h(t) = (t - a_1)^{n_1} \cdots (t - a_k)^{n_k}$ for some $n_i$ with 
		$n_1 + \cdots + n_k = 2g$.   In the partial fraction decomposition
		of the rational function $\frac{g(t)}{h(t)}$, if the residues at $t = a_1, \ldots, a_k$ are all $0$, then we have
		\begin{equation}
		\label{pfracdecomp}
		\frac{g(t)}{h(t)} = \sum_{i = 1}^k \left( \frac{c_{i,2}}{(t - a_i)^2} + \cdots + \frac{c_{i,n_i}}{(t - a_i)^{n_i}}\right ).
		\end{equation}
		Note that there are (only) $\sum (n_i - 1) = 2g - k$ coefficients $c_{i,j}$ here.  Now the equations 
		\eqref{RosenlichtDiff}, taking $f$ equal to each one of a collection of $g - 1$ of affine coordinate functions at $p$,
		yield $g - 1$ additional independent linear relations on the coefficients in \eqref{pfracdecomp}. As a result,  
		the space of solutions of \eqref{RosenlichtDiff} has dimension $\le g - (k - 1) < g$.  Since 
		$\dim {\rm H}^0(C,\omega_C) = g$, this shows that there 
		must be some abelian differentials with nonzero residues at points in $\nu^{-1}(p)$.  
		
		Since there must be $\frac{c}{t - a}$ terms in the partial fraction decomposition of some of the abelian differentials, 
		the Abel map from $C$ to $\CC^g$, and the parametrization of the theta divisor 
		contain logarithmic terms.  This contradicts the assumption that the theta divisor is algebraic, as in 
		Example~\ref{nonunibranch}.
	\end{proof}

	Our next task is to compute explicitly the algebraic theta functions arising from the curves in Proposition \ref{AlgCharac}.
	
	\subsection{Rational unibranch curves with a unique singularity}
	
	Let $C$ be an irreducible and reduced Gorenstein rational curve of genus $g$  with a unique unibranch singularity, and let $\nu\colon \mathbb{P}^1 \to C$ be its normalization. We also fix a smooth point $P_0 \in C$ and we choose an affine coordinate $u$ on $\mathbb{P}^1$ such that $P_0$ corresponds to $u=0$ and the singular point corresponds to $u=\infty$. In particular any canonical differential $\omega\in H^0(C,\omega_C)$ is a rational differential on $\mathbb{P}^1$ with poles only at $u = \infty$, so it is of the form $F(u)du$, with $F$ a polynomial. 
	
	In the notation of Section \ref{sec:thetatau}, we choose the divisor $D = (g-1)P_0$ and the corresponding line bundle $L=\omega_C(-(g-1)P_0)$ of degree $g-1$. Then we have a corresponding point $U=\iota(H^0(C,L(\infty P_0)))$ in the Sato Grassmannian, and thanks to Remark \ref{rmk:framecurve}, we can find a frame from differentials
	\begin{equation}\label{eq:frameonesingularity} 
	\omega_1 = F_1(u)du, \,\dots, \,\omega_g = F_g(u)du, \qquad \omega_{g+i} = \frac{1}{u^{i+1}}du \,\, \text{ for all }  i\geq 1 
	\end{equation}
	where the $\omega_1,\dots,\omega_g$ are a basis of $H^0(C,\omega_C)$ and the $F_i(u)$ are the corresponding polynomials. The tau function of this frame is a polynomial that can be described as follows:
	
	\begin{proposition}\label{lemma:shapetau}
		Let $\lambda_0,\lambda_{\infty}$ be the Weierstrass partitions of the curve $C$ at the smooth point $P_0$ and at the singular point respectively. Then the tau function of $U=\iota(H^0(C,L(\infty P_0)))$ has the form
		\[
		\tau = \xi_{\lambda_0}\cdot \sigma_{\lambda_0}+\sum_{\lambda_0 < \lambda' < \lambda_{\infty}} \xi_{\lambda'}\cdot \sigma_{\lambda'}+\xi_{\lambda_{\infty}}\cdot \sigma_{\lambda_{\infty}}
		\]
		with $\xi_{\lambda_0}\ne 0$ and $\xi_{\lambda_{\infty}} \ne 0$. In particular, it is a polynomial and the terms of lowest and highest weight are $\sigma_{\lambda_0}$ and $\sigma_{\lambda_{\infty}}$ respectively.
	\end{proposition}
	\begin{proof}
		First, let $\mathbb{W}^{\infty} = ( w^{\infty}_1,w^{\infty}_2,\dots,w^{\infty}_g )$ be the set of gaps at the singular point.  Since the singularity is Gorenstein, we know that $w^{\infty}_1=1,w^{\infty}_g=2g-1$. From the condition \eqref{RosenlichtDiff}, we see that if $\omega \in H^0(C,\omega_C)$ is a canonical differential then $\operatorname{ord}_{\infty}(\omega) \in \{-w^{\infty}_1-1,\dots,-w^{\infty}_g-1\}$. Hence, we can find a basis $\omega_i = F_i(u)du$ of $H^0(C,\omega_C)$ where $\deg F_i = w^{\infty}_i-1$. This yields the frame
		\[ \iota(\omega_i) = u^{2-g}F_i(u) \text{ for } i=1,\dots,g, \quad \iota(\omega_{g+i}) = u^{-g-i+1} \text{ for } i\geq 1 \]
		and if we write it in the form \eqref{eq:xishape} we obtain a matrix of the form
		
		\begin{equation}
		\label{eq:frameonesingularpoint}
		\xi = \begin{pmatrix}
		I & 0 \\
		0 & 0 \\
		0 & B \\
		0 & 0 \\
		\vdots & \vdots 
		\end{pmatrix}
		\end{equation}
		where $I$ is an infinite (in the upper left direction) identity matrix, and $B$ is a $w^{\infty}_g\times g$ matrix, which is between the rows $-(g-1)$ and $w^{\infty}_g-(g-1)$ of $\xi$. In particular, the $-g$ row of $\xi$ is zero. Furthermore, the matrix $B$ is in column echelon form, where the pivot of the column $i$ is at the row $w^{\infty}_i+1-g$. Then, from the shape of this matrix we can see that
		\[ \tau = \sum_{\lambda'<\lambda_{\infty}} \xi_{\lambda'}\cdot \sigma_{\lambda'} + \xi_{\lambda_{\infty}}\sigma_{\lambda_{\infty}}  \] 
		where $\xi_{\lambda_{\infty}}\ne 0$.
		
		In the same way, let $\mathbb{W}^{0} = ( w^{0}_1,w^{0}_2,\dots,w^{0}_g )$ be the gap sequence corresponding to the smooth point $P_0$. Then we can find another basis $\widetilde{\omega}_1,\dots,\widetilde{\omega}_g$ of $H^0(C,\omega_C)$ such that $\widetilde{\omega}_j = \widetilde{F}_j(u)du$, with $\operatorname{ord}_0(\widetilde{F}_j(u)) = w^0_j-1$. Then the tau function coming from the corresponding frame has the form
		\[ \widetilde{\tau} = \widetilde{\xi}_{\lambda_{0}}\sigma_{\lambda_{0}}+ \sum_{\lambda_{0}<\lambda'} \widetilde{\xi}_{\lambda'}\cdot \sigma_{\lambda'}   \] 
		where $\widetilde{\xi}_{\lambda_{0}}\ne 0$. Since two different frames for the same point in the Sato Grassmannian give the same tau function up to scalar multiplication, we can conclude.
	\end{proof}
	
	At this point, it is easy to show that the theta function and the tau function are essentially the same, up to a scalar multiplication. We fix one basis $\omega_1,\dots,\omega_g$ of $H^0(C,\omega_C)$ and we denote by $\tau(\mathbf{x})$ one tau function induced by the frame \eqref{eq:frameonesingularity}. We denote by $\theta(z_1,\dots,z_g)$ one theta function, coming from the Abel map with base point $P_0$ and basis of differentials $\omega_1,\dots,\omega_g$. We also let $A$ be the matrix appearing in \eqref{eq:equationtheta}. Observe that in this case this is actually a finite matrix. 
	
	\begin{theorem}\label{thm:tauthetaonesing}
		Let $C$ be a unibranch, Gorenstein rational curve with a unique singularity. Then with the above notation, we have that
		\[ \tau(\mathbf{x}) = \exp(c_0)\cdot \theta(A\mathbf{x}) \]
		for a certain $c_0\in \mathbb{C}$.
	\end{theorem}
	\begin{proof}
		The idea is to replicate the proof of Proposition \ref{prop:pullbacktheta} and Theorem \ref{thm:tauthetageneral}. Indeed, in our case, the local coordinate $u$ extends to a global coordinate on the whole of $C_0\cong \mathbb{C}$. Hence, repeating the proof of Proposition \ref{prop:pullbacktheta} with this global coordinate, we see that the divisor $a_n^*\overline{W}_{g-1}$ on $C_0^n$ has global equation $\tau_n(u_1,\dots,u_n)=\tau(p_1,\frac{1}{2}p_2,\dots,\frac{1}{g}p_g)$, which is a polynomial because of Proposition \ref{lemma:shapetau}. On the other hand, following the proof of Theorem \ref{thm:tauthetageneral}, the polynomial $\theta_n(u_1,\dots,u_n)=\theta(A\mathbf{p})$ is also a global equation for the same divisor. Hence it must be that $\tau_n = C_n \cdot \theta_n$ for a certain $C_n\in \mathbb{C}^*$. At this point, reasoning as in the proof of Theorem \ref{thm:tauthetageneral} we see that it must be $C_n=C_g$ for all $n\geq g$. Hence, taking $c_0$ such that $\exp(c_0)=C_g$ and reasoning as in the proof of Theorem \ref{thm:tauthetageneral}, it follows that $\tau(\mathbf{x})=\exp(c_0)\cdot \theta(A\mathbf{x})$.	
	\end{proof}
	
	Now we want to relate this tau function with the algebraic theta divisors we found before. We first exhibit an example as follows:
	
	\begin{example}
		\label{SatoThetaDivisorMonomial456}
		Consider the curve $C$ from Example~\ref{Monomial456}. The corresponding semigroup is encoded by the partition $\lambda = (4,1,1,1)$. We have seen in Example \ref{ThetaDivisorMonomial456} that the corresponding algebraic theta function, was
		\[ 
		\theta(z_1,z_2,z_3,z_4) = z_1^7+21z_1^4z_3-84z_1^3z_2^2+252z_1z_3^2+252z_2^2z_3-252z_4.
		\]
		The corresponding tau function is given exactly by a single Schur-Weierstrass polynomial
		$\tau = \sigma_{(4,1,1,1)} =\frac{1}{252}\left( x_1^7+21x_1^4x_3-84x_1^3x_2^2+252x_2^2x_3+252x_1x_3^2-252x_7 \right)$.
		Indeed, one can check that
		\[ \theta(x_1,x_2,x_3,x_7) = 252\cdot \tau \]
		which verifies the theorem. In particular, we can recover the tau function from the theta function. But we can also recover the theta function from the tau function: it is enough to set $x_i=0$ in $\tau$ for all $i\notin\{1,2,3,7\}$. In other words, we set to $0$ all the variables indexed by the semigroup of the singularity.
		$\triangle$
	\end{example}
	
	This example generalizes, recovering a result of \cite{BEL} about theta functions of monomial curves.
	
	\begin{corollary}\label{cor:monomial}
		Let $\mathbb{H}\subseteq \mathbb{N}$ be a Gorenstein semigroup and let $\mathbb{W}=(w_1,\dots,w_g)$ be the corresponding Weierstrass gap sequence and $\lambda$ the corresponding partition. Consider the monomial curve $C$ which is given as the image of
		\[ \nu\colon \mathbb{P}^1 \longrightarrow \mathbb{P}^{g-1}, \qquad (u:t) \mapsto (u^{w_1-1}t^{2g-1-w_1}: \ldots : u^{w_g-1}t^{2g-1-w_{g}}). \]
		Then this is a Gorenstein rational curve of genus $g$ with a unique unibranch singularity, with a basis of canonical differentials $(\omega_1,\omega_2,\dots,\omega_g)$ with $\omega_i = u^{w_i-1}du$.  Taking the tau function and the theta function with respect to this basis we obtain
		\[ \tau(\mathbf{x}) =  \sigma_{\lambda}(\mathbf{x}) =
		c\ \theta(x_{w_1},x_{w_2},\dots,x_{w_g}),\]
		for some nonzero constant $c$. 
	\end{corollary}
	\begin{proof}
		One can check explicitly that the curve $C$ is rational, unibranch, with a unique singularity at $\nu(1,0)$, which is Gorenstein with semigroup $\mathbb{W}$. Furthermore, one can check that the differentials $\omega_i = u^{w_i-1}du$ give a basis of the canonical differentials. Then, the argument in the proof of Proposition \ref{lemma:shapetau} shows that the Weierstrass partitions at $u=0$ and $u=\infty$ must both be equal to $\lambda$, and then Proposition \ref{lemma:shapetau} shows that it must be $\tau(\mathbf{x})=\sigma_{\lambda}(\mathbf{x})$ up to a scalar multiple. The other equality follows from Theorem \ref{thm:tauthetaonesing}.
	\end{proof}
	
	\begin{remark}
		In particular, this shows that the Schur-Weierstrass polynomial $\sigma_{\lambda}(\mathbf{x})$ depends only on the variables $x_{w_1},\dots,x_{w_g}$.
	\end{remark}
	
	\begin{example}\label{ex:semigroup<2,9>}
		For a different example, consider the curve $C$ given by the image of the map
		\begin{align*}
		\label{rattriple}
		\phi : \mathbb{P}^1 &\longrightarrow \mathbb{P}^3\\
		(t:u) &\longmapsto (t^6:t^4u^2:t^2u^4 + t^5u/3 :2t^3u^3/3 + u^6).
		\end{align*}
		This is an irreducible rational curve of arithmetic genus $4$ whose point 
		$P = \phi(0:1) = (0:0:0:1)$ is its only singularity.  It can be checked that the singularity is unibranch and Gorenstein with value semigroup equal to the hyperelliptic semigroup $\mathbb{S}_4 = \langle 2, 9 \rangle$. If, indeed, we write the 
		homogeneous coordinates in 
		$\mathbb{P}^3$ as $(x:y:z:w)$ then the coordinate function $z$ restricted to the curve vanishes to order $2$ at $P$, 
		$y$ vanishes to order $4$, 
		$x$ vanishes to order $6$, and 
		$xw - yz$ vanishes to order $9$.
		The corresponding partition is $\lambda_{\infty}=(4,3,2,1)$. Instead, the Weierstrass partition at the point $P_0=\phi(1:0) = (1:0:0:0)$ is $\lambda_0=(1,1,1,1)$. A basis for 
		${\rm H}^0(C,\omega_C)$ is given by 
		$\left( \left(u^6+\frac{2}{3}u^3\right)du,\left(u^4+\frac{1}{3}u\right)du,u^2du,du\right)$.
		Taking $\{u=0\}$ as the base point of the Abel map, the resulting algebraic theta function turns out to be
		\begin{align*}
		\theta(z_1,z_2,z_3,&z_4) = z_4^{10}-45z_3z_4^7-15z_4^7+315z_2z_4^5-4725z_3^3z_4+4725z_2z_3z_4^2-1575z_1z_4^3\\
		&-(175/4)z_4^4-4725z_3^2z_4+1575z_2z_4^2-4725z_2^2+4725z_1z_3-1050z_3z_4+1575z_1\nonumber 
		\end{align*}
		whereas the tau function turns out to be
		\[
		\tau(\mathbf{x})= -\frac{2}{9}\sigma_{(1,1,1,1)}(\mathbf{x})-\frac{1}{3}\sigma_{(4,1,1,1)}(\mathbf{x})-\frac{2}{3}\sigma_{(2,2,2,1)}(\mathbf{x}) +\sigma_{(4,3,2,1)}(\mathbf{x}).
		\]
		We can check explicitly that 
		\[ \tau(\mathbf{x}) = \frac{1}{4725}\, \theta\left( x_7+\frac{2}{3}x_4,x_5+\frac{1}{3}x_2,x_3,x_1\right),\]
		in agreement with Theorem \ref{thm:tauthetaonesing}. In particular, we can recover the tau function from the theta function. But we can also recover the theta function from the tau function: it is enough to set $x_i=0$ for all $i\notin\{1,3,5,7\}$ in $\tau$. In other words, we set to $0$ all the variables indexed by the semigroup of the singularity.  $\triangle$
	\end{example}

    This latest feature extends to all rational curves with a unique unibranch singularity.

\begin{corollary}\label{cor:thetafromtau}
	Let $C$ be a unibranch, Gorenstein rational curve with a unique unibranch singularity $P_{\infty}$ whose semigroup is $\mathbb{H}$. Then there is a basis of canonical differentials such that the corresponding theta function satisfies,
	\[ \exp(c_0) \cdot \tau(\mathbf{x})\big|_{x_h = 0\,\,  \mathrm{for\,\,all}\,\, h\in \mathbb{H}} =  \theta(x_{w_1^{\infty}},\dots,x_{w^{\infty}_g}), \]
	 for some nonzero constant $c_0$
\end{corollary}
\begin{proof}
With the previous notation, choose a basis $\omega_i = F_i(u)du$ of $H^0(C,\omega_C)$ such that the $F_i$ are monic of degree $\deg F_i = w_i^{\infty}-1$ and moreover the $F_j$ with $j\ne i$ have no terms of degree $w_i^{\infty}-1$.
	With our hypotheses, the matrix $A$ appearing in Theorem \ref{thm:tauthetaonesing} is a $g\times w^{\infty}_g$ matrix in reduced row echelon form, where the pivots correspond exactly to the $x_{w_i^{\infty}}$. Then the result follows from a straightforward computation. 
\end{proof}

	We have seen how the tau and theta functions are related. A related important notion connecting these functions is the so-called \emph{sigma function}. The generalized Klein sigma function has been introduced in \cite{{BEL}}.  This is a generalization of the elliptic Weierstrass sigma function, for the family of $(n,s)$ curves. This was later studied by Nakayashiki with a different approach \cite{Nak}. The sigma function differs from the Riemann theta function by an exponential factor, which becomes modular invariant. As in the case of the elliptic function, the Klein sigma function has a power series expansion in a neighborbood around ${\bf 0} \in \mathbb{C}^g$. It turns out that the constant term of this expansion is nothing but the Schur-Weierstrass polynomial up to a constant factor \cite{{BEL}}. It is worth to note that the Schur-Weierstrass polynomials associated with certain partitions give rise to rational solutions of the KP hierarchy \cite{AdlMos}. Also, \cite[Theorem 8]{Nak2009} asserts that the tau function can be written as the product of an exponential expression and the sigma function for the case of $(n,s)$ curves. Hence it is beneficial to study the sigma function further to understand the relation between the tau and the theta functions.

	\subsection{Rational unibranch curves with more than one singularity}

	In what follows, we will discuss the general case: Let $C$ be an  irreducible, Gorenstein rational curve of arithmetic genus $g$ with only unibranch singularities, and with normalization $\nu : \mathbb{P}^1 \longrightarrow C$. We choose an affine coordinate $t$ on $\mathbb{P}^1$ such that the points corresponding to the singularities are at $t=0$ and also at $t=t_1,\dots,t_s$. In particular, the point at infinity $P_0$ is a regular point, with a local coordinate $u=\frac{1}{t}$.  With this notation, the smooth locus of $C$ can be identified with
	\[ C_0 = \mathbb{C}_u \setminus \{ t_1^{-1},t_2^{-1},\dots,t_{s}^{-1} \} = \operatorname{Spec} \mathbb{C}[u]_{(1-t_1u)(1-t_2u)\dots (1-t_su)} \]
	and in particular, $u$ works as a global coordinate. We fix one basis $\omega_1,\dots,\omega_g$ of $H^0(C,\omega_C)$ and we denote by $\tau(\mathbf{x})$ one tau function induced by the frame \eqref{eq:frameonesingularity} for $H^0(C,L(\infty P_0))$, where $L=\omega_C(-(g-1)P_{0})$. We denote by $\theta(z_1,\dots,z_g)$ one theta function, coming from the Abel map with base point $P_0$ and basis of differentials $\omega_1,\dots,\omega_g$. We also let $A$ be the matrix appearing in \eqref{eq:equationtheta}. With this we can prove:
	
	\begin{theorem}\label{thm:tauthetamoresing}
		There is a constant  $c_0 \in \mathbb{C}$ and integers $c_1,\dots,c_s\in \mathbb{Z}$ such that
		\[ \tau(\mathbf{x}) = \exp\left( c_0 + \sum_{i=1}^{\infty} (c_1t_1^i+\dots+c_st_s^i)x_i \right) \cdot \theta(A\mathbf{x}) \]
		in the infinitely many variables $\mathbf{x} = (x_1,x_2,x_3,\dots)$.
	\end{theorem}
	\begin{proof}
		We follow the proof of Theorem \ref{thm:tauthetaonesing}, again exploiting the fact that the local coordinate $u$ extends to a global coordinate on the whole of $C_0$. Repeating the proof of Proposition \ref{prop:pullbacktheta} with this global coordinate, we see that the divisor $a_n^*\overline{W}_{g-1}$ on $C_0^n$ has global equation $\tau_n(u_1,\dots,u_n)=\tau(p_1,\frac{1}{2}p_2,\dots,\frac{1}{g}p_g)$, which is an element of $\mathbb{C}[u]_{(1-t_1u)\dots(1-t_su)}$. On the other hand, following the proof of Theorem \ref{thm:tauthetageneral}, the element $\theta_n(u_1,\dots,u_n)=\theta(A\mathbf{p})$ in $\mathbb{C}[u]_{(1-t_1u)\dots(1-t_su)}$ is also a global equation for the same divisor. Since $\mathbb{C}[u]_{(1-t_1u)\dots(1-t_su)}$ is a UFD, this means that $\tau_n$ and $\theta_n$ differ by an invertible element, meaning a nonzero scalar or a multiple of $(1-t_iu_k)$. So we can write
		\[ \tau_n(u_1,\dots,u_n) = c^{(n)}\cdot \prod_{h=1}^s \prod_{k=1}^n \left(\frac{1}{1-t_hu_k}\right)^{c^{(n)}_{h}} \theta_n(u_1,\dots,u_n) \]
		for a certain $c^{(n)}\in \mathbb{C}^*$ and $c^{(n)}_{h} \in \mathbb{Z}$. Observe that the $c^{(n)}_h$ do not depend on $k$ because both $\tau_n$ and $\theta_n$ are symmetric in $u_1,\dots,u_n$.
		At this point, reasoning as in the proof of Theorem \ref{thm:tauthetageneral} we see that it must be  $c^{(n)}=c^{(g)} = c$ and $c_h^{(n)}=c_h^{(g)} = c_h$ for all $n\geq g$. So we write
		\[ \tau_n(u_1,\dots,u_n) = c\cdot \prod_{h=1}^s \prod_{k=1}^n \left(\frac{1}{1-t_hu_k}\right)^{c_{h}} \theta_n(u_1,\dots,u_n).\]
		To conclude, we can use the expansion $\log\left(\frac{1}{1-tu}\right) = \sum_{i=1}^{\infty} t^i\frac{u^i}{i}$ to rewrite the above as
		\begin{equation}\label{tauTheta} \tau_n(u_1,\dots,u_n) = \exp\left( \log c + \sum_{i=1}^{\infty} (c_1t_1^i+\dots+c_st_s^i)\frac{p_i}{i} \right) \theta_n(u_1,\dots,u_n)
		\end{equation}
		which implies the final result, thanks to Lemma \ref{lemma:equality}.	
	\end{proof}
	
	In particular, this shows that these algebraic theta functions give rational solutions to the KP equation.
	
	\begin{corollary}\label{cor:rationalsolutionkp}
		With the notations of Theorem \ref{thm:tauthetamoresing}, let $\mathbf{U},\mathbf{V},\mathbf{W}$ be the first three columns of the matrix $A$. Then the function
		\[ f(x,y,t) := 2\frac{\partial^2}{\partial x^2} \log \theta(\mathbf{U}x+\mathbf{V}y+\mathbf{W}t) \]
		is a rational solution to the KP equation.
	\end{corollary}
	\begin{proof}
		Since $\tau(\mathbf{x})$ is a solution to the KP hierarchy, the function $2\frac{\partial^2}{\partial x^2} \log \tau(x,y,t,0,0,0,\dots)$ is a solution to the KP equation. However,  Theorem~\ref{thm:tauthetamoresing} shows
		this is exactly equal to the function $u(x,y,t)$. Moreover, since $\theta$ is a polynomial, it follows that $u(x,y,t)$ is a rational function. 
	\end{proof}
	
	\begin{example}
		\label{twosings}
		Consider the curve $C$ which is the image of the map
		\[ \nu\colon \mathbb{P}^1 \longrightarrow \mathbb{P}^3, \qquad (u:t) \mapsto (t^4(t-u)^2:ut^3(t-u)^2:u^4(t-u)^2:t^6). \]
		This is an integral, Gorenstein, unibranch curve of arithmetic genus $4$. It has two singular points: at $u=\infty$ with semigroup $\langle 3,4\rangle$ and at $u=1$ with semigroup $\langle 2,3\rangle$. A basis of the canonical  differentials is given by
		$\left( du, u\,du, u^4du,\frac{1}{(1-u)^2}du\right)$, and if we take the smooth point at $u=0$ as the base point for the Abel map, a corresponding theta function is
		\begin{align}
		\label{alttheta}
		\theta(z_1,z_2,z_3,z_4) = &(z_1^5-5z_1^4+10z_1^3-20z_1z_2^2-10z_1^2+20z_1z_2+20z_2^2-20z_2+20z_3)z_4\\
		&+3z_1^5-10z_1^4+10z_1^3-60z_1z_2^2+20z_1z_2+40z_2^2+60z_3.
		\nonumber
		\end{align}
		We can also compute the corresponding tau function as:
		\begin{equation} 
		\tau(\mathbf{x}) = \sum_{\lambda_1\geq 3} (\lambda_1+3)\sigma_{(\lambda_1,3,1,1)} - 4\sigma_{(2,2,1,1)}-3\sigma_{(2,1,1,1)}
		\end{equation}
		and with some manipulations one can arrive at
		\[\tau(\mathbf{x}) = \exp\left( -\log 20+\sum_{i=1}^{\infty} x_i \right) \cdot \theta\left(x_1,x_2,x_5,\sum_{i=1}^{\infty}ix_i\right), \]
		which agrees with Theorem \ref{thm:tauthetamoresing}. $\triangle$
	\end{example}

	We close this section by adding a few words on how to compute the integers $c_1,\dots,c_s$ appearing in~\eqref{tauTheta}. It is enough to consider the case $n=g$. Then the computation in the proof of Proposition \ref{prop:pullbacktheta} should show that
	\begin{align*} 
	\tau_g(u_1,\dots,u_g) =  \prod_{h=1}^s \prod_{k=1}^n \left(\frac{1}{1-t_hu_k}\right)^{2\delta_{h}} \widetilde{\tau}_g(u_1,\dots,u_g), 
	\end{align*}
	where $\delta_h$ is the delta-invariant of the singularity at $t_h$, and $\widetilde{\tau}_g$ is a polynomial that is not divisible by any of the $(1-\tau_hu_k)$. Then we can also write
	\[
	\theta_g(u_1,\dots,u_g) = \prod_{h=1}^s \prod_{k=1}^n \left(\frac{1}{1-t_hu_k}\right)^{c'_{h}} \widetilde{\theta}_g(u_1,\dots,u_g)
	\]
	for some nonnegative integers $c_h'$ and a polynomial $\widetilde{\theta}_g$ that is not divisible by any of the $(1-t_hu_k)$. Then we see that we can compute the $c_h$ as $c_h = 2\delta_h - c_h'$.

	\section{Degrees of the algebraic theta divisors}\label{sec:algSurf}
	
	In this section, we prove Theorem \ref{theoremDeg} from the Introduction.
	Before giving the proof, we will consider two examples.
	There are (non-hyperelliptic) curves with singular points having the value semigroups
	$\mathbb{S}_g=\langle 2,2g+1 \rangle$ for all arithmetic genera $g \ge 3$.  We start with an example of an
	algebraic theta function coming from such a curve of arithmetic genus 4, one larger than in Eiesland's tacnode-cusp example.
	
	\begin{example}
		\label{example:hyperellipticnonmonomial}
		
		The curve in Example~\ref{ex:semigroup<2,9>} satisfies the hypotheses of Theorem~\ref{AlgCharac} and its algebraic theta divisor attains the degree bound for $g = 4$.  $\triangle$ 

	\end{example}
	
	Our second example will illustrate what happens when there is more than one singular point. The key idea is that there are local contributions from each singular point 
	that determine the total degree of the algebraic theta function.  
	
	\begin{example}
		\label{Genus4TwoSingularPoints}
		Consider the curve $C$ from Example~\ref{twosings}. 
		We have seen that the theta divisor has equation
		\begin{align*}
		\theta(z_1,z_2,z_3,z_4) = &(z_1^5-5z_1^4+10z_1^3-20z_1z_2^2-10z_1^2+20z_1z_2+20z_2^2-20z_2+20z_3)z_4\\
		&+3z_1^5-10z_1^4+10z_1^3-60z_1z_2^2+20z_1z_2+40z_2^2+60z_3.
		\end{align*}
		The partitions corresponding to the two singular points are $(3,1,1)$ and $(1)$.
		We will see that this explains both why the total degree of the 
		implicit equation is 6 and why there is a unique monomial of highest degree, which factorizes as $z_1^5 z_4$.  $\triangle$
		
		\vskip 10pt
		
		We will actually prove a more precise version of Theorem \ref{theoremDeg} as follows: let $C$ be a rational and unibranch Gorenstein curve and suppose that it has singularities at $P_1,\dots,P_h$. Let $\delta_j = \delta(P_j)$ be the corresponding delta invariants and let $\mathbb{W}^j = (w^j_{\delta_j},\dots,w^j_{1})$ be the Weierstrass sequences of $P_j$. Then we can take a basis $\omega^j_i$ for $j=1,\dots,h$ and $i=1,\dots,\delta(P_j)$ such that on the normalization $\mathbb{P}^1$ the differential $\omega^j_i$ has a pole of order $w^j_i+1$ at the point corresponding to $P_j$, and has no other poles. We denote the corresponding coordinates on the Jacobian $\operatorname{Jac}(C)$ by $z^j_i$ and we group them together as $\mathbf{z} = (\mathbf{z}^1,\mathbf{z}^2,\dots,\mathbf{z}^h)$. Recall that $\lambda_{P_j}$ denotes the associated partition to the gaps in $\mathbb{W}^j$. 
		
		\begin{theorem}\label{thm:degreeprecise}
			Let $\theta(\mathbf{z})$ be the equation, up to translation, of a theta divisor obtained from the above basis and a smooth base point $P_0 \in C$. Then $\deg_{\mathbf{z}^j} \theta(\mathbf{z}) = |\lambda_{P_j}|$ and $\deg_{\mathbf{z}}\theta(\mathbf{z}) = \sum_{j=1}^h |\lambda_{P_j}|$. Even more precisely, we can write
			\[ \theta(\mathbf{z}) = (z^1_1)^{|\lambda_{P_1}|}\cdot (z^2_1)^{|\lambda_{P_2}|} \dotsb (z^h_1)^{|\lambda_{P_h}|} + (\text{terms of lower degree in each set of variables } \mathbf{z}^j). \]  
		\end{theorem}

		
		
		\begin{proof} 
		Our proof of Theorem~\ref{thm:degreeprecise} will be by induction on the number of singular points.
			Suppose first that $C$ is monomial with a unique unibranch singularity $P_1$, and let $\lambda_{P_1}$ be the corresponding partition. Then Corollary \ref{cor:monomial} shows that, after the substitution $z^1_i = x_{w_i^{\infty}}$ the theta function coincides, up to a nonzero scalar, with the Schur-Weierstrass polynomial $\sigma_{\lambda}(\mathbf{x})$, which is a polynomial of degree $|\lambda_{P_1}|$, whose unique term of highest degree is $x_1^{|\lambda_{P_1}|}$. Hence, we can write the theta function as
			\[ \theta(z^1_1,\dots,z^1_g) = (z^1_1)^{|\lambda|} + (\text{ terms of lower degree }). \] 
			The same holds for general curves with a unique unibranch singularity, thanks to Corollary \ref{cor:thetafromtau} and Proposition \ref{lemma:shapetau}. However, this case follows also from the general strategy that we are going to present now.

				For the case that the curve $C$ has more than one singular point, or has one singularity that is non-monomial we will proceed by induction on the number of singularities. Thus, let $P_1,\dots,P_h$ be the singular points of $C$. On the normalization $\mathbb{P}^1$ we choose an coordinate $u$ such that $P_1$ is at $u=\infty$ and the $P_i$ are at $u=\frac{1}{t_i}$ for $i=2,\dots,h$.  Then we can find a basis of differentials of the form
				\begin{align}\label{eq:basisdiffC}
				\omega^{1}_i &= \left( u^{w^{1}_i-1} + \sum_{k=0}^{w^1_i-2}a^1_{i,k} u^k \right)du &  \text{ for } i=1,\dots,\delta_1,  \\ 
				\omega^{j}_i &=  \left( \frac{1}{(u-t_j)^{w^j_i+1}} + \sum_{k=2}^{w^j_i} \frac{a^j_{i,k}}{(u-t_j)^k} \right) du& \text{ for } j=2,\dots,h \text{ and } i=1,\dots,\delta_j,   
				\end{align}
				with some complex numbers $a_{i,k}^j$. Now the idea is to degenerate this curve to the union of one curve $C_1$ with one singularity corresponding to $P_1$ and another curve $C_2$ that has singularities corresponding to $P_2,\dots,P_h$. Furthermore, the singularity on $C_1$ will be monomial, with the same semigroup as the one of $P_1$. We do it taking inspiration from \cite[Section III]{Fay} and \cite{Teissier}: take two smooth rational curves $\widetilde{C}_1$ and $\widetilde{C}_2$ with affine coordinates $u_1,u_2$ respectively.
				Now choose $c\in \mathbb{C}$ such that $|t_i|\leq |c|$ for each $i=2,\dots,h$ and consider for every $s\in \mathbb{C}^*$ the two open subsets $U_1 =\{ |u_1| < |\frac{s}{c}| \} \subseteq C_1$ and $U_2 = \{ |u_2| > |c| \} \subseteq C_2$. We can glue them together via the identification $u_1=\frac{s}{u_2}$, and this gives a family $\widetilde{\mathcal{C}}$ of smooth rational curves $\widetilde{{C}}_s$ over $\mathbb{C}^*$ that degenerate to the union of two smooth rational curves meeting at a node. Following \eqref{eq:basisdiffC}, consider the meromorphic differentials on $\widetilde{{C}}_s$ given by
				\begin{align*}\label{eq:basisdiffCs}
				\omega^{1}_i(s) &= \left( u_1^{w^{1}_i-1} + \sum_{k=0}^{w^1_i-2}a^1_{i,k} s^{w^1_i-1-k} u_1^k \right)du_1 = - s^{w^1_i}\left( \frac{1}{u_2^{w^1_i+1}} + \sum_{k=0}^{w^1_2-2}a^1_{i,k} \frac{1}{u_2^{k+2}} \right) du_2,  \\ 
				\omega^{j}_i(s) &=  \left( \frac{1}{(u_2-t_j)^{w^j_i+1}} + \sum_{k=2}^{w^j_i} \frac{a^j_{i,k}}{(u_2-t_j)^k} \right) du_2 = -s\left( \frac{u_1^{w^j_i-1}}{(s-t_ju_1)^{w^j_i+1}} + \sum_{k=2}^{w^j_i} \frac{a^j_{i,k}u_1^{k-2}}{(s-t_ju_1)^k} \right) du_1  .   
				\end{align*}
				We see that for $s=1$ we recover exactly the canonical differentials of our original curve $C$, and moreover under the action of $\mathbb{C}^*$ on $(u_1,u_2,s)$ given by $\lambda\cdot (u_1,u_2,s) = (\lambda u_1,u_2,\lambda s)$ the differentials transform as
				\begin{equation}\label{eq:torusaction}
				\lambda \cdot \omega^1_i(s) = \lambda^{w^{1}_i} \omega^1_i, \qquad \lambda \cdot \omega^j_i(s) = \omega^j_i(s) \text{ for } j=2,\dots,h.
				\end{equation}
				Hence, the image of the map
				$\widetilde{\mathcal{C}} \to \mathbb{P}^{g-1} \times \mathbb{C}^*$ induced by $\omega^i_j(s)$ is a family $\widetilde{\mathcal{C}}$ of curves ${C}_s$ in which
				${C}_1$ coincides with $C$, all ${C}_s$ for $s \ne 0$
				are isomorphic to $C$,  and ${C}_s$ degenerates to reducible curve $C_0=C_1\cup C_2$ when $s = 0$.
				
				
				Now, consider the base point $P_0 = \{ u_1 =0 \} = \{ u_2= \infty \}$ on each curve $\mathcal{C}_s$ for $s\ne 0$. We obtain a family of theta divisors $\Theta_s \subseteq \mathbb{C}^g \times \mathbb{C}^*_s$ which we can write as
				\begin{equation}\label{eq:thetas}
				\theta'(\mathbf{z},s) = \theta'_n(\mathbf{z})s^n + \theta'_{n-1}(\mathbf{z})s^{n-1} + \dots + \theta'_1(\mathbf{z})s + \theta'_0(\mathbf{z}).
				\end{equation}
				In particular $\theta'(\mathbf{z},1)$ recovers the equation of the theta divisor for our original curve $C$. First we want to study the term $\theta_0'(z)$:
				
				\begin{lemma}\label{lemma:degeneration}
					We have $\theta'_0(\mathbf{z}) = \theta_1(\mathbf{z}^1)\cdot \theta_2(\mathbf{z}^2,\dots,\mathbf{z}^h)$, where $\theta_1(\mathbf{z}^1)$ and $\theta_2(\mathbf{z}^2,\dots,\mathbf{z}^h)$ are equations of theta divisors of the curves $C_1,C_2$, coming from the bases in \eqref{eq:basisdiffC}. 
				\end{lemma}	
			    \begin{proof}
			    	We want to compute the abelian integrals $\int_{P_0}^{Q(s)} \omega^j_i(s)$ for a family of smooth points $Q(s)$ according to whether the limit $Q(0)$ belongs to $C_1$ or $C_2$. In the first case, we can represent $Q(s)$ via the coordinate $u_1(s)$ and we can compute 
			    	\[ \lim_{s\to 0}\int_{P_0}^{Q(s)} \omega^1_i = \frac{1}{w^1_i} u_1(0)^{w^1_i}, \qquad \lim_{s\to 0}\int_{P_0}^{Q(s)} \omega^j_i = \frac{1}{w^j_i} \left( -\frac{1}{t_j} \right)^{w^j_i} + \sum_{k=2}^{w^j_i} \frac{a^1_{i,k}}{k-1}\left( -\frac{1}{t_j} \right)^{k-1}  \]
			    	whereas if $Q(0)$ belongs to $C_2$ we can represent $Q(s)$ via $u_2(s)$ and we get
			    	\[ \lim_{s\to 0}\int_{P_0}^{Q(s)} \omega^1_i = 0, \qquad \lim_{s\to 0}\int_{P_0}^{Q(s)} \omega^j_i = \frac{1}{w^j_i} \frac{1}{(u_2(0)-t_j)^{w^j_i}} + \sum_{k=2}^{w^j_i}\frac{a^j_{i,k}}{(u_2(0)-t_j)^{k-1}}.  \]
			    	Hence, suppose we have families $Q_1(s),\dots,Q_a(s)$ of smooth points with limits in $C_1$ and other families of smooth points $Q_{a+1}(s),\dots,Q_{g-1}(s)$ with limits in $C_2$. Then we can write the limits of the Abel maps with base point $P_0$ as
			    	\begin{equation} 
			    	\lim_{s\to 0} \sum_{i=1}^{g-1} \int_{P_0}^{Q_i(s)} \boldsymbol{\omega}  = \left( \sum_{i=1}^{a} \int_{0}^{Q_i(0)} \boldsymbol{\omega}_1, \sum_{i=a+1}^{g-1} \int_0^{Q_i(0)} \boldsymbol{\omega}_2 + D_a \right)
			    	\end{equation}
			    	where $\boldsymbol{\omega}$ is a basis of $H^0(\mathcal{C}_s,\omega_{\mathcal{C}_s})$, $\boldsymbol{\omega}_i$ is a basis of $H^0(C_i,\omega_{C_i})$, and $D_a$ is a constant depending only on $a$.
			    	In other words, the images $\alpha_{\mathcal{C}_s}(\mathcal{C}_s^{(g-1)}) \subseteq \mathbb{C}^g$ of the Abel maps corresponding to the curve $\mathcal{C}_s$ degenerate to $\bigcup_{a=0}^{g-1} \alpha_{C_1}(C_1^{(a)}) \times (\alpha_{C_2}(C_2^{(g-1-a)})+D_a)$. At this point, we observe that $g=g(C_1)+g(C_2)$ and moreover $\alpha_{C_i}(C_i^{n})$ is dominant onto $\mathbb{C}^{g(C_i)}$ as soon as $n\geq g(C_i)$. Hence, taking closures, it follows that the theta divisors $\Theta_s \subseteq \mathbb{C}^g$ degenerate to the theta divisor $\Theta_0 = (\Theta_1 \times \mathbb{C}^{g(C_2)}) \cup (\mathbb{C}^{g(C_1)} \times \Theta_2)$, where $\Theta_1$ is the theta divisor of $C_1$ arising from the basis in \eqref{eq:basisdiffC} and $\Theta_2$ is a translate of the theta divisor of $C_2$ arising from the basis in \eqref{eq:basisdiffC}. In other words, if $\theta_s(\mathbf{z})$
			    	is the equation of the family of theta divisors $\Theta_s$, then $\theta_0'(\mathbf{z}) = \theta_1(\mathbf{z}_1) \cdot \theta_2(\mathbf{z}_2)$, where $\theta_i(\mathbf{z}_i)$ is an equation for $\Theta_i \subseteq \mathbb{C}^{g(C_i)}$. 
			    \end{proof}

                Then we show that $\theta'_0(\mathbf{z})$ computes the degree with respect to $\mathbf{z}^1$. 
		        
				\begin{lemma}\label{lemma:bounddegree}
					We have that $\deg_{\mathbf{z}^1} \theta'_{\ell}(\mathbf{z}) \leq \deg_{\mathbf{z}^1}\theta'_0(\mathbf{z})-\ell$ for all $\ell=0,\dots,n$.
				\end{lemma}
			    \begin{proof}
			    The torus action \eqref{eq:torusaction} shows that if we give the weights $\operatorname{wt}(s)=1$ to $s$, $\operatorname{wt}(z_{w^1_i}) = w^1_i$ to the variables in $\mathbf{z}^1$ and finally weight $0$ to all the variables in $\mathbf{z}^2,\dots,\mathbf{z}^h$, the polynomial \eqref{eq:thetas} is isobaric with respect to this weight, meaning that $\operatorname{wt}\theta_i(\mathbf{z}) = \operatorname{\theta}'_0(\mathbf{z})-i$. Finally, we look at the degrees with respect to the variables $\mathbf{z}^1$ of the polynomials $\theta'_i(\mathbf{z})$ appearing in \eqref{eq:thetas}. We first observe that
			    \begin{align*} 
			    \deg_{\mathbf{z}^1} \theta_0'(\mathbf{z}) &=  \deg_{\mathbf{z}^1} \theta_1(\mathbf{z}_1)\theta_2(\mathbf{z}_2) = \deg_{\mathbf{z}^1} \theta_1(\mathbf{z_1}) = \operatorname{wt} \theta_1(\mathbf{z}_1) = \operatorname{wt}\theta'_0(\mathbf{z}) 
			    \end{align*}
			    where the equality $\deg_{\mathbf{z}^1} \theta_1(\mathbf{z}^1) = \operatorname{wt} \theta_1(\mathbf{z}^1)$ comes from the fact that with the substitution $z^1_{w^1_j} = x_{w^1_j}$ the polynomial $\theta_1(\mathbf{x})$ is a Schur-Weierstrass polynomial for the partition $\lambda_{P_0}$ (recall that the singularity on $C_1$ is monomial) hence both degree and weight are equal to $|\lambda_{P_0}|$. Hence, it follows that
			    \[ \deg_{\mathbf{z}^1} \theta'_{i}(\mathbf{z}) \leq \operatorname{wt} \theta'_i(\mathbf{z}) = \operatorname{wt} \theta_0'(\mathbf{z}) - i = \deg_{\mathbf{z^1}} \theta'_0(\mathbf{z})-i.   \]	
			    \end{proof}
                
                At this point, we can conclude the proof of Theorem  \ref{thm:degreeprecise}. Indeed, first we observe that the equation for the theta divisor of our original curve is given by $\theta(\mathbf{z}) =\theta'(\mathbf{z},1)$, hence
                \[ \theta(\mathbf{z}) = \theta_n'(\mathbf{z})+\dots+\theta'_0(\mathbf{z}). \]
                Moreover, Lemma \ref{lemma:degeneration} tells us that $\theta'_0(\mathbf{z})$ is the product of two equations $\theta_1(\mathbf{z}^1)$ and $\theta_2(\mathbf{z}^2,\dots,\mathbf{z}^h)$ corresponding to the theta divisors of two curves with a strictly smaller number of singularities. Hence, applying the induction hypothesis of Theorem \ref{thm:degreeprecise} to these factors, we see that
                \[ \theta'_0(\mathbf{z}) = (z^1_1)^{|\lambda_{P_1}|}\cdot (z^2_1)^{|\lambda_{P_2}|} \dotsb (z^h_1)^{|\lambda_{P_h}|} + (\text{terms of lower degree in each set of variables } \mathbf{z}^j) \]
                and then applying Lemma \ref{lemma:bounddegree} we see that we can write
                \[ \theta(\mathbf{z}) = (z^1_1)^{|\lambda_{P_1}|}\cdot (z^2_1)^{|\lambda_{P_2}|} \dotsb (z^h_1)^{|\lambda_{P_h}|} + (\text{terms of lower degree in the variables } \mathbf{z}^1)  \]
                Now the observation is that we can repeat the whole reasoning with another set of coordinates on $\mathbb{P}^1$ such that the point $P_i$ is at $u=\infty$, and this allows us to write for each $j=1,\dots,h$
                 \[ \theta(\mathbf{z}) = (z^1_1)^{|\lambda_{P_1}|}\cdot (z^2_1)^{|\lambda_{P_2}|} \dotsb (z^h_1)^{|\lambda_{P_h}|} + (\text{terms of lower degree in the variables } \mathbf{z}^j)  \]
                 which concludes the proof.

		\end{proof}
		
		\begin{proof}[Proof of Theorem \ref{theoremDeg}] Theorem \ref{thm:degreeprecise} immediately implies most of Theorem \ref{theoremDeg} and to conclude the proof  we just need to prove the last statement. We know that the degree of the algebraic theta divisor of $C$ is $\sum_{j=1}^h |\lambda_{P_j}|$ and since the singularities are Gorenstein, we know that $\lambda_{P_j}$ is a subpartition of $(\delta_j-1,\delta_j-2,\dots,0)$. Hence we see that
		    \begin{equation}\label{bound}
		  \sum_{j=1}^h |\lambda_{P_j}| \leq \sum_{j=1}^h\frac{\delta_j(\delta_j+1)}{2} \leq \frac{g(g+1)}{2},
		    \end{equation}
		    where the last inequality follows from the fact that $g=\sum_{j=1}^h\delta_j$. In particular, we see that the bounds in \eqref{bound} are equalities if and only if $h=1$ and $\lambda_{P_1}=(g-1,g-2,\dots,1,0)$, meaning exactly that the semigroup is $\langle 2,2g+1\rangle$.     
		\end{proof}

	\end{example}

	\section{Further discussion on reducible curves}
	
	To conclude, we present some results and examples on reducible curves.  Suppose $C=C_1\cup \dots \cup C_m$ is a connected, reducible, Gorenstein curve of arithmetic genus $g$. For any divisor $D$ supported on the smooth part $C_0$,  its multidegree $\mathbf{d}=(d_1,d_2,\dots,d_m)$ is the collection of the degrees of its restriction to the various components. Then the symmetric product splits naturally according to the multidegrees: for any multidegree $\mathbf{d} = (d_1,d_2,\dots,d_m)$ we define $C_0^{(\mathbf{d})}$ as the set of effective divisors of multidegree $\mathbf{d}$. Then it is clear that $C_0^{(g-1)} = \bigcup_{\mathbf{d}} C_0^{(\mathbf{d})}$, where $d_1+\dots+d_m = g-1$. Hence, the theta divisor will have different components, but we can still ask whether it is algebraic or not, meaning that all components are algebraic. This is answered in the following generalization of Theorem~\ref{AlgCharac}.
	
	\begin{proposition} 
		\label{reduciblechar}
		Let $C$ be a possibly reducible, connected and reduced Gorenstein curve of arithmetic genus $g$. Then the theta divisor is algebraic if and only if every irreducible component $C_i$ of $C$ is rational with unibranch singularities, and  moreover the curve is tree-like, meaning that $H_1(C,\mathbb{Z})=0$. 
	\end{proposition}
	
	\begin{proof}
		This is as in the proof of Theorem \ref{AlgCharac}. The theta divisor is algebraic if and only if $\Lambda_C=0$, meaning that $\operatorname{Jac}(C) = H^0(C,\omega_C)^*$ is unipotent. Now we refer to the notation of Remark \ref{rmk:sequencejacobians}. If the curve satisfies our conditions, then we first see that $\widetilde{C}$ is a disjoint union of projective lines, hence $\operatorname{Jac}(\widetilde{C})=0$. At this point, our conditions and \cite[Proposition 10]{BLR} imply that $\operatorname{Jac}(C')=0$, where $C'$ is the partial normalization of $C$. 
		Finally, \cite[Proposition 9]{BLR} shows that $\operatorname{Jac}(C)$ is unipotent. 
		
		Conversely, suppose that $\Lambda_C=0$. Then it must be that $\Lambda_{\widetilde{C}}=0$ as well. Since $\widetilde{C}$ is a smooth curve, this implies that it is a union of projective lines, hence every component of $C$ is rational. 
		Hence, $\operatorname{Jac}(\widetilde{C})=0$ and then \cite[Proposition 10]{BLR} shows that $\operatorname{Jac}(C')$ is a torus. But it is also unipotent because $\Lambda_{C'}=0$. Thus, it must be trivial, and \cite[Proposition 10]{BLR} again proves that this is possible only if $C$ is tree-like and every component has all unibranch singularities.
	\end{proof}
	
	Now we give some examples:
	
	\begin{example}
		\label{CuspPlusTangent}
		We consider the reducible plane quartic $C$ given by the affine equation $y(y^2 - x^3) = 0$. This is a cuspidal cubic $C_1$ together with the cuspidal tangent line $C_2$. Eiesland also considers curves of this form in \cite{Eies09}.   The normalization, written component by component, is 
		\begin{align*}
		\nu : \mathbb{P}^1 \sqcup \mathbb{P}^1 &\longrightarrow C = C_1 \cup C_2\\
		(u:t) &\longmapsto (u^3:ut^2:t^3) \in C_1 \\
		(r:s) &\longmapsto (r:s:0) \in C_2 
		\end{align*}
		and then we get the following canonical differentials, which we write as pairs, showing the restrictions to the two components of the normalization:
		\begin{equation*}
		\omega_1  = \left(u du,-dr\right), \quad
		\omega_2 = \left(du,0\right), \quad
		\omega_3 = \left( u^3du,-rdr\right).
		\end{equation*}
		We will use the points at $u=0$ and $r=0$ as the base points for the Abel map. We consider the three possible multidegrees $(2,0),(1,1)$ and $(0,2)$. For the multidegree $(2,0)$ we have the map
		\[
		(u_1,u_2) \mapsto \left( \int_0^{u_1} udu + \int_0^{u_2} udu , \int_0^{u_1} du + \int_0^{u_2} du ,\int_0^{u_1} u^3du + \int_0^{u_2} u^3du  \right)
		\]
		and the corresponding algebraic theta divisor is given by $\theta(z_1,z_2,z_3) = z_2^4-4z_1z_2^2-4z_1^2+8z_3$. For the multidegree $(1,1)$ we have the map  
		\[
		(u_1,r_1) \mapsto \left( \int_0^{u_1} udu - \int_0^{r_1} dr , \int_0^{u_1} du + \int_0^{r_1} 0\, dr ,\int_0^{u_1} u^3du - \int_0^{r_1} rdr  \right)
		\]
		and the corresponding algebraic theta divisor is given by $\theta(z_1,z_2,z_3) = z_2^4+4z_1z_2^2-4z_1^2-8z_3$. Finally, the multidegree $(0,2)$ gives the map
		\[
		(r_1,r_2) \mapsto \left( - \int_0^{r_1} dr- \int_0^{r_2} dr , \int_0^{r_1} 0\,dr+\int_0^{r_2} 0 \, dr ,- \int_0^{r_1} rdr - \int_0^{r_2} rdr \right)
		\]
		whose hypersurface is simply given by $\theta(z_1,z_2,z_3)=z_2$.$\triangle$
	\end{example}	
	
	\begin{example}
		\label{GenusFourNonAlgebraic}
		This example shows why the condition of $C$ being tree-like is necessary.
		Consider the reducible curve $C = \{xy - z^2= z(wy - x^2) =0 \} \subset \mathbb{P}^3$.  This is the complete intersection of a quadric and a cubic, hence a Gorenstein canonical curve of degree 6 and arithmetic genus 4.  It is easy to see that the decomposition of $C$ into irreducible components is 
		$C = C_1 \cup C_2 \cup C_3$
		where $C_1 = \{x =z =0\}$ and  $C_2 = \{y =z =0\}$ are lines meeting at $(0:0:0:1)$ 
		while $C_3 = \{xy - z^2 = wy - x^2 \}$ is a rational quartic curve of arithmetic genus 1 with a cusp at $(0:0:0:1)$. The curves $C_1$ and $C_3$ intersect at two points, namely $(0:1:0:0),(0:0:0:1)$ and this leads to a closed chain which is not homologous to $0$. Thus the curve is not tree-like. And indeed, if we write the normalization as
		$\widetilde{C} = \mathbb{P}^1 \sqcup \mathbb{P}^1 \sqcup \mathbb{P}^1$
		and take affine coordinates $r,s,t$ on the three components, where $r = 0, s = 0, t = 0$ all map to the point $(0:0:0:1)$,
		then one of the sections of ${\rm H}^0(C,\omega_C)$ can be written as $({dr}/{r}, {ds}/{s}, {-2 dt}/{t})$.
		Hence, the  Abel map will contain logarithmic terms.  $\triangle$ 
	\end{example}
	
	
	
	\begin{example}
		\label{RedEx}
		Consider the Gorenstein canonical curve $C$ of degree 6 and arithmetic genus 4 in $\mathbb{P}^3$ defined by the homogeneous equations $\{xy - zw = (y-x)(y-x+z)(y-x-z) = 0\}.$
		The curve $C$ has three components, a parabola $C_1$ in the plane defined by $\{y - x = 0\}$, and two 
		hyperbolas $C_2,C_3$ in the planes defined by $\{y - x + z = 0\}$ and
		$\{y - x - z = 0\}$.  These three components meet only at the point, 
		$(w : x : y : z) = (1 : 0 : 0 : 0)$ in affine 3-space, so topologically $C$ is a ``bouquet'' of 3 two-spheres and it has trivial first homology ${\rm H}_1(C,\mathbb{Z})$.  By Proposition~\ref{reduciblechar} we expect to obtain algebraic theta divisors for all choices of multidegrees.  We compute a basis for ${\rm H}^0(C,\omega_C)$ using \eqref{RosenlichtDiff} and we find:
		\begin{equation*}
		\left(dr, -ds, 0\right),\quad
		\left(dr, 0, -dt\right),\quad
		\left(2rdr, -sds, -tdt\right), \quad
		\left(\left(2r + 2r^2\right)dr, \left(-2s -s^2\right)ds, -t^2dt\right),
		\end{equation*}
		where $r,s,t$ are local coordinates on the three components such that $r = s = t = \infty$ gives the point $(1:0:0:0)$.
		
		We choose the points with $u=s=t=0$ as base points of the Abel map, and we get  different hypersurfaces for each multidegree. For example, with the multidegree $(1,1,1)$ we get the hypersurface in $\mathbb{C}^4$ as
		\[  z_1^4+4z_1^3z_2-6z_1^2z_2^2+4z_1z_2^3+z_2^4-12z_1^2z_2+12z_1z_2^2+24z_1z_3+12z_3^2-12z_1z_4-12z_2z_4 = 0. \]
		If instead we take the multidegree $(2,1,0)$ we obtain the hypersurface
		\[
		2z_1^3-3z_1^2z_2-z_2^3-3z_1^2+6z_1z_2-3z_2^2+6z_2z_3+6z_3-6z_4 = 0.
		\]
		All the other hypersurfaces have degree $3$ or $4$ as well.		$\triangle$
	\end{example}
	
	Notice that both of the hypersurfaces computed in Example~\ref{RedEx} have degree significantly smaller than the maximal degree 10 for
	algebraic theta divisors from irreducible curves of arithmetic 
	genus $g = 4$ from Theorem~\ref{theoremDeg}.  The same is true for the remaining two cases not considered in the example, which both have degree 3. 
	The sum of the degrees is 13, however, and that is greater than the maximal degree for algebraic theta divisors from irreducible curves of genus $g = 4$.
	Based on additional computational evidence we have generated in other examples like this, we would make the following conjecture.
	
	\begin{conjecture}  
		The bound in Theorem~\ref{theoremDeg} also holds for all irreducible components of the theta divisors constructed from reducible curves of arithmetic genus $g$.
	\end{conjecture}
	
	We do not see a way to prove this at present, though, since Example~\ref{RedEx} shows that degenerating irreducible canonical curves to a reducible curve may not preserve the total degree of the components of the theta divisor when the theta divisor becomes reducible.

	\bigskip \medskip

	\noindent
	\footnotesize 
	{\bf Authors' addresses:}
	
	\smallskip
	
	\noindent Daniele Agostini, MPI-MiS Leipzig,
	\hfill  {\tt daniele.agostini@mis.mpg.de}
	
	\noindent 
	T\"urk\"u \"Ozl\"um \c{C}elik, Bo\u{g}azi\c{c}i University,
	\hfill {\tt turkuozlum@gmail.com}
	
	\noindent John B. Little,
	College of the Holy Cross,
	\hfill {\tt jlittle@holycross.edu}

\end{document}